\documentclass[10pt]{amsart}

\usepackage{enumerate, amsmath, amsfonts, amssymb, amsthm, mathtools, thmtools, wasysym, graphics, graphicx, xcolor, frcursive,xparse,comment,bbm}
\usepackage[all]{xy}
\usepackage[colorlinks=true, pdfstartview=FitV, linkcolor=blue, citecolor=blue, urlcolor=blue]{hyperref}

\usepackage{url, hypcap}
\hypersetup{colorlinks=true, citecolor=darkblue, linkcolor=darkblue}
\usepackage{mathrsfs}

\usepackage{tikz}
\usetikzlibrary{arrows,automata}
\usetikzlibrary{patterns,snakes}
\usetikzlibrary{shapes.geometric,calc}
\tikzstyle{arrow}=[thick, ->, >=stealth]
\tikzset{
	edge/.style={->,> = latex'}
}
\usetikzlibrary{calc,through,backgrounds,shapes,matrix}
\usepackage{graphicx}

\usepackage[letterpaper,margin=1.25in]{geometry}

\definecolor{darkblue}{rgb}{0.0,0,0.7} 
\definecolor{darkgreen}{rgb}{0, .6, 0} 
\definecolor{lightblue}{rgb}{0,135,147}

\definecolor{red}{rgb}{1,0,0}

\definecolor{darkred}{rgb}{0.7,0,0} 
\definecolor{lightgrey}{rgb}{0.7,0.7,0.7} 

\newtheorem{theorem}{Theorem}[section]
\newtheorem{proposition}[theorem]{Proposition}

\newtheorem{lemma}[theorem]{Lemma}

\theoremstyle{definition}
\newtheorem{definition}[theorem]{Definition}
\newtheorem{example}[theorem]{Example}

\newtheorem{remark}[theorem]{Remark}

\numberwithin{equation}{section}

\definecolor{darkred}{rgb}{0.7,0,0} 
\newcommand{\defn}[1]{{\color{darkred}\emph{#1}}} 

\usepackage[colorinlistoftodos]{todonotes}

\title[Mixing time]{Upper bounds on mixing time of finite Markov chains}

\author[J.~Rhodes]{John Rhodes}
\address[J. Rhodes]{Department of Mathematics, University of California, Berkeley, CA 94720, U.S.A.}
\email{rhodes@math.berkeley.edu, blvdbastille@gmail.com}

\author[A.~Schilling]{Anne Schilling}
\address[A. Schilling]{Department of Mathematics, UC Davis, One Shields Ave., Davis, CA 95616-8633, U.S.A.}
\email{anne@math.ucdavis.edu}

\date{\today}

\dedicatory{Dedicated to Ron Graham and Vaughan Jones}

\begin{document}

\begin{abstract}
We provide a general framework for computing upper bounds on mixing times of finite Markov chains when its minimal ideal is left zero. 
Our analysis is based on combining results by Brown and Diaconis with our previous work on stationary distributions of 
finite Markov chains. Stationary distributions can be computed from the Karnofsky--Rhodes and McCammond 
expansion of the right Cayley graph of the finite semigroup underlying the Markov chain. Using loop graphs, which are
planar graphs consisting of a straight line with attached loops, there are rational expressions for the stationary distribution
in the probabilities. From these we obtain bounds on the mixing time. In addition, we provide a new Markov chain on linear 
extension of a poset with $n$ vertices, inspired by but different from the promotion Markov 
chain of Ayyer, Klee and the last author. The mixing time of this Markov chain is $O(n \log n)$.
\end{abstract}

\maketitle

\section{Introduction}

A \defn{Markov chain} is a model that describes transitions between states in a state space according to certain probabilistic 
rules. The defining characteristic of a Markov chain is that the transition from one state to another only depends on the current
state and the elapsed time, but not how the system arrived there. In other words, a Markov chain is  ``memoryless''. 
Markov chains have an abundance of applications, from data analysis, population dynamics to traffic models.

For a Markov chain, the \defn{stationary distribution} $\Psi$ is the long-term limiting distribution. Mathematically
speaking, it is the eigenvector of the transition matrix $T$ of the Markov chain with eigenvalue one. That is
\[
	T \Psi = \Psi.
\]
An important question is how quickly does the Markov chain converge to the stationary distribution.
In Markov chain theory, distance is usually the total variation distance or half the $L^1$-norm in classical analysis. 
If $\Omega$ is the state space, the total variation distance between two probability distributions $\nu$ and $\mu$ is 
defined as
\[
	\|\nu - \mu\| = \max_{A \subseteq \Omega} |\nu(A) - \mu(A)|.
\]
For a given small $\epsilon>0$, the \defn{mixing time} $t_\mathsf{mix}$ is the smallest $t$ such that
\[
	\| T^t \nu - \Psi \| \leqslant \epsilon,
\]
independent of the initial distribution $\nu$.

In seminal work of Bidigare, Hanlon and Rockmore~\cite{BHR.1999}, which was continued by Diaconis, Brown, 
Athanasiadis, Bj\"orner, Chung and Graham, amongst others~\cite{BrownDiaconis.1998,BBD.1999,Brown.2000,Bjorner.2008,
Bjorner.2009,Athanasiadis.Diaconis.2010,ChungGraham.2012,Saliola.2012}, the special family of semigroups, now known as 
\emph{left regular bands} first studied by Sch\"utzenberger~\cite{Schuetzenberger.1947} in the forties, was applied to 
random walks or Markov chains on hyperplane arrangements. In his 1998 ICM lecture~\cite{Diaconis.icm.1998}, Diaconis 
discussed these developments. In Section~4.1, entitled \emph{What is the ultimate generalization?}, he asks 
how far the semigroup techniques can be taken. 

Every finite state Markov chain $\mathcal{M}$ has a random letter representation, that is, a representation of a semigroup 
$S$ acting on the left on the state space $\Omega$. See for example~\cite[Proposition 1.5]{LevinPeres.2017} 
and~\cite[Theorem 2.3]{ASST.2015}. In this setting, there is a transition $s \stackrel{a}{\longrightarrow} s'$ with probability 
$0\leqslant x_a\leqslant 1$, where $s, s'\in \Omega$, $a\in S$ and $s'=a.s$ is the action of $a$ on the state $s$. 
It is enough to consider the semigroup $S$ generated by the elements $a$ with $x_a>0$, called the generating set $A$.
For example, the Markov chain with state space $\Omega =\{1,2\}$ and transition diagram
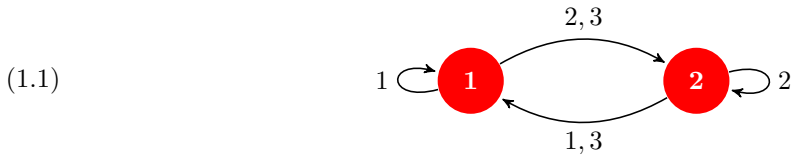
\begin{equation}
\label{equation.markov linear}
\raisebox{-1cm}{
\begin{tikzpicture}[->,>=stealth',shorten >=1pt,auto,node distance=3cm,
                    semithick]
  \tikzstyle{every state}=[fill=red,draw=none,text=white]

  \node[state]         (I) {$\mathbf{1}$};
  \node[state]         (A) [right of=I] {$\mathbf{2}$};
  \path (I) edge [bend left] node {$2,3$} (A)
               edge [loop left] node {$1$} (I)
           (A) edge[bend left] node {$1,3$} (I)
                 edge [loop right] node {$2$} (A);
\end{tikzpicture}}
\end{equation}
can be associated to the semigroup with right Cayley graph depicted in Figure~\ref{figure.right cayley linear}. 
The conceptual reason why a Markov chain described using the \textit{left} action of a semigroup can be analyzed using the 
\textit{right} Cayley graph is that if time goes left (due to the left action), then coupling from the past corresponds to
right multiplication. The transition matrix in this Markov chain is
\[
	T = \begin{pmatrix}
	x_1 & x_1+x_3\\
	x_2+x_3 & x_2
	\end{pmatrix}.
\]
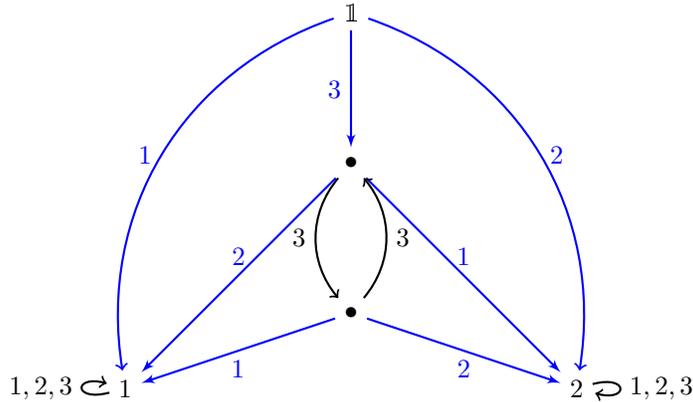
\begin{figure}[t]
\begin{center}
\begin{tikzpicture}[auto]
\node (A) at (0, 0) {$\mathbbm{1}$};
\node (B) at (-3,-5) {$1$};
\node(C) at (3,-5) {$2$};
\node(D) at (0,-2) {$\bullet$};
\node(E) at (0,-4) {$\bullet$};
\path (A) edge[->,thick, blue, bend right = 40] node[midway,left] {$1$} (B)
              edge[->,thick, blue, bend left = 40] node[midway,right] {$2$} (C);
\draw[edge,thick,blue] (A) -- (D) node[midway, left] {$3$};
\draw[edge,thick,blue] (D) -- (B) node[midway, above] {$2$};
\draw[edge,thick,blue] (D) -- (C) node[midway, above] {$1$};
\draw[edge,thick,blue] (E) -- (B) node[midway, below] {$1$};
\draw[edge,thick,blue] (E) -- (C) node[midway, below] {$2$};
\path (D) edge[->,thick, bend right = 40] node[midway,left] {$3$} (E);
\path (E) edge[->,thick, bend right = 40] node[midway,right] {$3$} (D);
\path (B) edge[->,thick, loop left] node {$1,2,3$} (B);
\path (C) edge[->,thick, loop right] node {$1,2,3$} (C);
\end{tikzpicture}
\end{center}
\caption{\label{figure.right cayley linear} The right Cayley graph $\mathsf{RCay}(S,A)$ of the semigroup that
gives the Markov chain in~\eqref{equation.markov linear} with generators $A=\{1,2,3\}$.}
\end{figure}

In the pursuit of finding Diaconis' ultimate generalization~\cite{Diaconis.icm.1998}, the arguments in Brown and 
Diaconis~\cite{BrownDiaconis.1998} were generalized to Markov chains for $\mathscr{R}$-trivial 
semigroups~\cite{ASST.2015}. In~\cite{RhodesSchilling.2019,RhodesSchilling.2019a}, the current authors developed
a general theory for computing the stationary distribution for any finite Markov chain. The theory uses semigroup methods
such as the Karnofsky--Rhodes and McCammond expansion of a semigroup. These expansions give rise to loop graphs
which immediately yield Kleene expressions for all paths from the root of the graph to elements in the minimal ideal of
the semigroup. The Kleene expressions in turn give rational expressions for the stationary distribution.

In this paper we apply the findings of~\cite{RhodesSchilling.2019,RhodesSchilling.2019a} to study upper bounds on the mixing
time of the Markov chain. In particular, Theorems~\ref{theorem.main} and~\ref{theorem.main1} provide upper bounds 
for the mixing time directly from the rational expression of the stationary distribution in the case when the minimal ideal
of the semigroup is left zero. This general theory is applied to specific examples (Tsetlin library, edge flipping on a line Markov chain,
and a new Markov chain on linear extensions) in Section~\ref{section.examples}.

The paper is organized as follows. In Section~\ref{section.mixing time}, we develop the main theory.
In Section~\ref{section.mixing truncation}, we present our main theorems regarding the upper bounds
on the mixing time (see Theorems~\ref{theorem.main} and~\ref{theorem.main1}). 
We discuss the relation to the Shannon entropy in Section~\ref{section.shannon}.
In Section~\ref{section.decreasing statistics}, we refine bounds on the mixing time using certain statistics
that were developed in~\cite{ASST.2015,ASST.2015a}. In Section~\ref{section.syntactic}, we consider semigroups
syntactic at zero. In particular, we prove in Theorem~\ref{theorem.syntactic} that the upper bounds on the mixing time do not change by
replacing the semigroup by its syntactic image. In Section~\ref{section.languages}, we relate observations on mixing
time to $d$-testable languages. Finally, in Section~\ref{section.examples} we consider specific examples such as the
Tsetlin library~\cite{Tsetlin.1963}, edge flipping on a line~\cite{BrownDiaconis.1998,ChungGraham.2012}, and a new Markov
chain on linear extensions of a poset with $n$ vertices, which is inspired by but different from the promotion Markov 
chain~\cite{AyyerKleeSchilling.2014}. This new Markov chain has a mixing time of $O(n \log n)$ as
compared to the mixing time of the model of Bubley and Dyer~\cite{BubleyDyer.1999} with mixing time $O(n^3 \log n)$.

\subsection*{Acknowledgments}
We are grateful to Arvind Ayyer, Darij Grinberg, John Hunter, Stuart Margolis, Igor Pak, Dan Romik, 
Eric Severson, Benjamin Steinberg, and Andrew Waldron for discussions.
The last author was partially supported by NSF grants DMS--1760329, DMS--1764153, and  DMS--205335. 
This material is based upon work supported by the Swedish Research Council under
grant no. 2016-06596 while the author was in residence at Institut 
Mittag--Leffler in Djursholm, Sweden during Spring 2020.

An extended abstract of this paper has appeared in the proceedings for FPSAC 2021~\cite{RS.2021}.

\section{Mixing time}
\label{section.mixing time}

Let $T$ be the \defn{transition matrix} of a finite Markov chain. Assuming that the Markov chain is \defn{ergodic}
(meaning that it is irreducible and aperiodic), by the Perron--Frobenius Theorem there exists a unique \defn{stationary 
distribution} $\Psi$ and $T^t \nu$ converges to $\Psi$ as $t\to \infty$ for any initial state $\nu$. A Markov chain is irreducible 
if the graph of the Markov chain is strongly connected. It is aperiodic if the gcd of the cycle lengths in the graph of the Markov 
chain is one. In fact, the stationary distribution is the right eigenvector of eigenvalue one of $T$
\[
	T \Psi = \Psi.
\]
The \defn{mixing time} measures how quickly the Markov chain converges to the stationary distribution.
For a given small $\epsilon>0$, $t_\mathsf{mix}$ is the smallest $t$ such that
\[
	\| T^t \nu - \Psi \| \leqslant \epsilon.
\]

We begin this section by reviewing methods to compute upper bounds on mixing times in Section~\ref{section.upper bound}
(see in particular Theorem~\ref{theorem.ASST}), relations between ideals and semaphore codes and how this relates to mixing time
in Section~\ref{section.semaphore}, and the Markov and Chernoff inequalities to bound mixing time in Sections~\ref{section.markov}
and~\ref{section.chernoff}. The semigroup methods of~\cite{RhodesSchilling.2019,RhodesSchilling.2019a} to compute rational expressions
of the stationary distribution of a Markov chain in terms of the probabilities $x_a$ for the generators $a\in A$ of the semigroup
are reviewed in Section~\ref{section.rational}. Our main new results for the upper bounds of the mixing times in terms of truncations
of the rational expressions of the stationary distribution (Theorem~\ref{theorem.main}) and using a Cauchy--Euler operator (Theorem~\ref{theorem.main1})
are stated in Section~\ref{section.mixing truncation}. In Section~\ref{section.shannon} we discuss the relation between Shannon entropy and
mixing time. Sections~\ref{section.decreasing statistics}-\ref{section.languages} are devoted to new results in special settings, for example
for monoids which are syntactic at zero (Theorem~\ref{theorem.syntactic}) and $d$-testable languages (Remark~\ref{remark.ideal containment}).

\subsection{Upper bound}
\label{section.upper bound}

Brown and Diaconis~\cite{BrownDiaconis.1998}~\cite[Theorem 0]{Brown.2000} showed, for Markov chains
associated to left regular bands, that the total variational distance from stationarity after $t$ steps is bounded above
by the probability $\mathsf{Pr}(\tau> t)$, where $\tau$ is the first time that the walk hits a certain ideal. The arguments in
Brown and Diaconis~\cite{BrownDiaconis.1998} can be generalized to arbitrary finite Markov chains (not just those
related to left regular bands). To state the details, we need some more notation.

Let $\mathcal{M}(S,A)$ be a finite state Markov chain with state space $\Omega$ and transition matrix $T$ associated to the 
semigroup $S$ with generators $A$ with probabilities $0<x_a\leqslant 1$ for $a\in A$.

A two-sided \defn{ideal} $I$ (or ideal for short) is a subset $I \subseteq S$ such that $u I v \subseteq I$ for all 
$u,v \in S^{\mathbbm{1}}$, where $S^{\mathbbm{1}}$ is the semigroup $S$ with identity $\mathbbm{1}$ added (even if
$S$ already contains an identity). If $I,J$ are ideals of $S$, then $IJ \subseteq I \cap J$, so that $I \cap J \neq \emptyset$. 
Hence every finite semigroup has a unique nonempty minimal ideal denoted $K(S)$.

Assume that the minimal ideal $K(S)$ is \defn{left zero}, that is, $xy=x$ for all $x,y\in K(S)$.
This assumption implies that the Markov chain on the minimal ideal (given by the left action) is ergodic.
Let $\tau$ be the random variable which is the time that the random walk is absorbed into the minimal ideal $K(S)$.

\begin{theorem} \cite{ASST.2015}
\label{theorem.ASST}
Let $S$ be a finite semigroup whose minimal ideal $K(S)$ is a left zero semigroup and let $T$ be the transition
matrix of the associated Markov chain. Then
\[
	\| T^t \nu - \Psi \| \leqslant \mathsf{Pr}(\tau>t).
\]
\end{theorem}

\begin{proof}
By~\cite[Corollary 3.5(3)]{ASST.2015}, we have
\[
	\| T^t \nu - \Psi \| \leqslant P^{\star t}(S \setminus K(S)),
\]
where $P^{\star n}$ denotes the $n$-th convolution power of $P$.
By~\cite[Eq. (4.6)]{ASST.2015}, the right hand side equals $\mathsf{Pr}(\tau>t)$.
\end{proof}

\subsection{Ideals and semaphore codes}
\label{section.semaphore}
Let $A$ be a finite alphabet, $A^+$ the set of all nonempty words in the alphabet $A$, and $A^\star$ the set of all
words in the alphabet $A$.

As shown in~\cite{RSS.2016}, ideals in $A^+$ are in bijection with semaphore codes~\cite{BPR.2010}. 
A \defn{prefix code} is a subset of $A^+$ such that all elements are incomparable in prefix order (meaning that no 
element is the prefix of any other element of the code). 
A \defn{semaphore code} $\mathcal{S}$ is a prefix code such that $A \mathcal{S} \subseteq \mathcal{S} A^\star$.
There is a natural left action on a semaphore code. If $u \in \mathcal{S} \subseteq A^+$ and $a \in A$, then $a u$ 
has a prefix in $\mathcal{S}$ (and hence a unique prefix of $a u$). The left action $a.u$ is the prefix of $a u$ that is in 
$\mathcal{S}$. Assigning probability $0 \leqslant x_a \leqslant 1$ to $a\in A$, the left action on a semaphore code 
$\mathcal{S}$ defines a Markov chain with a countable state space $\mathcal{S}$.

The bijection between ideals $I \subseteq A^+$ and semaphore codes $\mathcal{S}$ over $A$ is given as follows
(see~\cite[Proposition 4.3]{RSS.2016}).
If $u = a_1 a_2  \ldots a_j \in I \subseteq A^+$, find the (necessarily unique) index $1\leqslant i \leqslant j$ 
such that $a_1 \ldots a_{i-1} \not \in I$, but $a_1 \ldots a_i \in I$. Then $a_1 \ldots a_i$ is a code word and the set of all 
such words forms the semaphore code $\mathcal{S}$. Conversely, given a semaphore code 
$\mathcal{S}$, the corresponding ideal is $\mathcal{S} A^\star$. 

In this setting, $\tau$ can be interpreted as the random variable given by the length of the semaphore code words. 
Let $\mathcal{S}$ be a semaphore code and $I$ the ideal under the
bijection described above. A semaphore code word $s = s_1 s_2 \ldots s_\ell$ has the property that $s\in I$,
but $s_1 s_2 \ldots s_{\ell-1} \not \in I$. Hence $\tau$ can be interpreted as the random variable
given by the length $\ell$.

Next we discuss two ways to approximate $\mathsf{Pr}(\tau>t)$ using Markov's and Chernoff's inequality.

\subsection{Markov's inequality}
\label{section.markov}
By Markov's inequality (see for example~\cite{LevinPeres.2017,DevroyeLugosi.2001}), we have
\begin{equation}
\label{equation.Markov inequality}
	\mathsf{Pr}(\tau>t) \leqslant \frac{E[\tau]}{t+1},
\end{equation}
where $E[\tau]$ is the expected value for $\tau$, the first time the walk hits the ideal. We have
\begin{equation}
\label{equation.Etau}
	E[\tau] = \sum_{a=1}^\infty \mathsf{Pr}(\tau \geqslant a).
\end{equation}

\subsection{Chernoff's inequality}
\label{section.chernoff}
Chernoff's inequality uses the moment generating function combined with Markov's 
inequality~\eqref{equation.Markov inequality} to give an upper bound on $\mathsf{Pr}(\tau\geqslant t)$. More precisely,
\[
	\mathsf{Pr}(\tau\geqslant t) = \mathsf{Pr}(e^{s\tau} \geqslant e^{st}) \qquad \text{for $s>0$.}
\]
Hence by Markov's inequality~\eqref{equation.Markov inequality}
\[
	\mathsf{Pr}(\tau \geqslant t) \leqslant \frac{E[e^{s\tau}]}{e^{st}}
\]
and since this is true for all $s>0$
\[
	\mathsf{Pr}(\tau \geqslant t) \leqslant \min_{s>0}\left\{ \frac{E[e^{s\tau}]}{e^{st}} \right\}.
\]

\subsection{Rational expressions for stationary distributions}
\label{section.rational}

Let $\mathcal{M}(S,A)$ be the Markov chain associated to the finite semigroup $S$ with generators in $A$. 
Assume that its minimal ideal $K(S)$ is left zero, so that $K(S)$ can be taken as the state space $\Omega$ of the 
Markov chain. Denote by $\mathcal{S}(S,A)$ the semaphore code associated to $K(S)$ (see 
Section~\ref{section.semaphore}). For a word $s \in A^+$, we denote by $[s]_S$ the image of the word in the alphabet 
$A$ in $S$. The following theorem is stated in~\cite[Corollaries 2.23 \& 2.28]{RhodesSchilling.2019}.

\begin{theorem} \cite{RhodesSchilling.2019}
\label{theorem.stationary}
If $K(S)$ is left zero, the stationary distribution of the Markov chain $\mathcal{M}(S,A)$ labeled by $w \in K(S)$
is given by
\begin{equation}
\label{equation.Psiw}
	\Psi_w(x_1,\ldots,x_n) = \sum_{\stackrel{s \in \mathcal{S}(S,A)}{[s]_S=w}} 
	\; \prod_{a\in s} x_a.
\end{equation}
\end{theorem}

In~\cite{RhodesSchilling.2019,RhodesSchilling.2019a}, we developed a strategy using loop graphs to compute 
the expressions in Theorem~\ref{theorem.stationary} as rational functions in the probabilities $x_a$ for $a\in A$. 
This is done in several steps:
\begin{enumerate}
\item We used the McCammond and Karnofsky--Rhodes expansion $\mathsf{Mc}\circ \mathsf{KR}(S,A)$ of the right 
Cayley graph $\mathsf{RCay}(S,A)$ of the semigroup $S$ with generators $A$. In this paper we do not require the details
of these definitions, except that the right Cayley graph as well as its expansions are rooted graphs with root $\mathbbm{1}$.
The Karnofsky--Rhodes expansion is another right Cayley graph, whereas the McCammond expansion is only an 
automata. For the precise definition of the Karnofsky--Rhodes expansion, we refer the reader 
to~\cite[Definition 4.15]{MRS.2011}, \cite[Section 3.4]{MSS.2015}, \cite[Section 2.4]{RhodesSchilling.2019a},
and~\cite[Section 2]{RSS.2020}. For the definition of the McCammond expansion, we refer the reader 
to~\cite[Section 2.7]{MRS.2011} and~\cite[Section 2.5]{RhodesSchilling.2019a}. The Markov chain $\mathcal{M}(S,A)$
is a \defn{lumping}~\cite{LevinPeres.2017} of the Markov chains associated to the expansions.
\item The stationary distributions of the Markov chains associated to the expansions can be expressed using
\defn{loop graphs} $G$, see~\cite{RhodesSchilling.2019a}. A loop graph is a straight line path from $\mathbbm{1}$ to an
endpoint $s$ with directed loops of any finite length attached recursively to any vertex (besides $\mathbbm{1}$ and $s$).
In this way~\cite[Theorem 1.4]{RhodesSchilling.2019a}
\begin{equation}
\label{equation.lumping}
	\Psi_w(x_1,\ldots,x_n) = \sum_G \Psi_G(x_1,\ldots,x_n),
\end{equation}
where the sum is over certain loop graphs $G$ with end point $s$ such that $[s]_S=w$. 
Here~\cite[Definition 1.3]{RhodesSchilling.2019a}
\begin{equation}
\label{equation.PsiG}
	\Psi_G(x_1,\ldots,x_n) = \sum_p \; \prod_{a \in p} x_a,
\end{equation}
where the sum is over all paths $p$ in $G$ starting at $\mathbbm{1}$ and ending in $s$.
\item There is a \defn{Kleene expression} for the set of all paths from $\mathbbm{1}$ to $s$ in $G$. The Kleene expression 
immediately yields a rational expression for the stationary distribution $\Psi_G(x_1,\ldots,x_n)$ and hence 
$\Psi_w(x_1,\ldots,x_n)$ by~\eqref{equation.lumping}.
\end{enumerate}

\begin{remark}
\label{remark.series expansion}
An important property of the above construction is that in the series expansion of the rational expression for
$\Psi_w(x_1,\ldots,x_n)$ (resp. $\Psi_G(x_1,\ldots,x_n)$) the total degree of each term corresponds to the length of
the underlying semaphore code word in~\eqref{equation.Psiw} (resp. the underlying path in $G$ in~\eqref{equation.PsiG}).
\end{remark}

\subsection{Mixing time via truncation of Kleene expressions}
\label{section.mixing truncation}

As stated in Theorem~\ref{theorem.ASST}, $\mathsf{Pr}(\tau \geqslant t)$ provides an upper bound on the mixing time
in the setting that $K(S)$ is left zero. As discussed in Section~\ref{section.semaphore}, $\tau$ can be interpreted
as the random variable given by the length of the semaphore code words or paths in the loop graph. To compute 
$\mathsf{Pr}(\tau \geqslant t)$, one needs to compute the sum of probabilities of all paths of length weakly greater than $t$. 
By Remark~\ref{remark.series expansion}, the length of the paths is given by the total degree in the probability variables
$x_1,\ldots,x_n$ for the generators $a_1,\ldots,a_n$ of the semigroup $S$.
Hence we obtain $\mathsf{Pr}(\tau\geqslant t)$ by truncating the rational function for the stationary distribution to total degree
weakly bigger than $t$.

Let $\Psi^{\geqslant t}_w(x_1,\ldots,x_n)$ be the truncation of the formal power series associated to the rational function
$\Psi_w(x_1,\ldots,x_n)$ to terms of degree weakly bigger than $t$ and let $\Psi_w^{<t}(x_1,\ldots,x_n)$ be the truncation
of the formal power series associated to the rational function $\Psi_w(x_1,\ldots,x_n)$ to terms of degree strictly smaller
than $t$. Note that
\[
	\Psi_w(x_1,\ldots,x_n) = \Psi_w^{<t}(x_1,\ldots,x_n) + \Psi^{\geqslant t}_w(x_1,\ldots,x_n).
\]

\begin{theorem}
\label{theorem.main}
Suppose the Markov chain satisfies the conditions of Theorem~\ref{theorem.ASST}.
If $\Psi_w(x_1,\ldots,x_n)$ is represented by a rational function such that each term of degree $\ell$ in its formal power 
sum expansion corresponds to a semaphore code word $s$ of length $\ell$ with $[s]_S=w$, we have
\[
	\mathsf{Pr}_w(\tau \geqslant t) = \frac{\Psi^{\geqslant t}_w(x_1,\ldots,x_n)}{\Psi_w(x_1,\ldots,x_n)}
	= 1 - \frac{\Psi_w^{<t}(x_1,\ldots,x_n)}{\Psi_w(x_1,\ldots,x_n)}.
\]
\end{theorem}

For each $w\in K(S)$, we can also give an explicit formula for the expected number of steps $E_w[\tau]$ it takes to reach
the endpoint of $w$ using the Cauchy--Euler operator.

\begin{theorem}
\label{theorem.main1}
Suppose the Markov chain satisfies the conditions of Theorem~\ref{theorem.ASST}.
If $\Psi_w(x_1,\ldots,x_n)$ is represented by a rational function such that each term of degree $\ell$ in its formal power 
sum expansion corresponds to a semaphore code word $s$ of length $\ell$ with $[s]_S=w$, we have
\[
	E_w[\tau] = \left( \sum_{i=1}^n x_i \frac{\partial}{\partial x_i} \right) \ln \Psi_w(x_1,\ldots,x_n).
\]
\end{theorem}

\begin{remark}
Note that the formal expression for $\Psi_w(x_1,\ldots,x_n)$ cannot be manipulated using that $x_1+\cdots+x_n=1$
when using Theorems~\ref{theorem.main} and~\ref{theorem.main1}.
\end{remark}

\begin{proof}[Proof of Theorem~\ref{theorem.main1}]
Let the formal power sum expression for the rational function $\Psi_w(x_1,\ldots,x_n)$ be as follows
\[
	\Psi_w(x_1,\ldots,x_n) = \sum_{m_1,\ldots,m_n\geqslant 0} c_{m_1,\ldots,m_n} x_1^{m_1}\cdots x_n^{m_n}.
\]
Then formally
\begin{multline*}
	\left( \sum_{i=1}^n x_i \frac{\partial}{\partial x_i} \right) \ln \Psi_w(x_1,\ldots,x_n)
	= \frac{\left( \sum_{i=1}^n x_i \frac{\partial}{\partial x_i} \right) \Psi_w(x_1,\ldots,x_n)}{\Psi_w(x_1,\ldots,x_n)}\\
	= \frac{\sum_{m_1,\ldots,m_n\geqslant 0} c_{m_1,\ldots,m_n}(m_1+\cdots+m_n) x_1^{m_1} \cdots x_n^{m_n}}
	{\sum_{m_1,\ldots,m_n\geqslant 0} c_{m_1,\ldots,m_n} x_1^{m_1} \cdots x_n^{m_n}}.
\end{multline*}
Note that a term $x_1^{m_1} \cdots x_n^{m_n}$ of degree $m_1+\cdots+m_n$ corresponds to a semaphore code word
of length $m_1+\cdots+m_n$. Hence $c_{m_1,\ldots,m_n} (m_1+\cdots+m_n) 
x_1^{m_1}\cdots x_n^{m_n}/ \Psi_w(x_1,\ldots,x_n)$
is the length of the path times the probability of having taken a path with $m_i$ steps along the $i$-th generator.
The sum over all such terms is precisely $E_w[\tau]$.
\end{proof}

\begin{remark}
Let $\mathsf{Pr}_G(\tau \geqslant t)$ be the probability that the length of the paths in the loop graph $G$ from
$\mathbbm{1}$ to the end point $s$ is weakly bigger than $t$. Then by analogous argument as above, we also have
\begin{equation}
\label{equation.PrG}
	\mathsf{Pr}_G(\tau \geqslant t) = \frac{\Psi^{\geqslant t}_G(x_1,\ldots,x_n)}{\Psi_G(x_1,\ldots,x_n)}
	= 1 - \frac{\Psi_G^{<t}(x_1,\ldots,x_n)}{\Psi_G(x_1,\ldots,x_n)}
\end{equation}
and
\begin{equation}
\label{equation.EG}
	E_G[\tau] = \left( \sum_{i=1}^n x_i \frac{\partial}{\partial x_i} \right) \ln \Psi_G(x_1,\ldots,x_n).
\end{equation}
\end{remark}

\begin{example}[Single loop]
\label{example.single loop}
Suppose the path in the loop graph $G$ from $\mathbbm{1}$ to the ideal is a straight line with a single loop
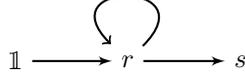
\begin{center}
\begin{tikzpicture}[auto]
\node (A) at (0, 0) {$\mathbbm{1}$};
\node (B) at (1.5,0) {$r$};
\node(C) at (3,0) {$s$};
\draw[edge,thick] (A) -- (B);
\draw[edge,thick] (B) -- (C);
\path (B) edge [thick,loop]  (B);
\end{tikzpicture}
\end{center}
where the loop is taken with probability $p$ and the step to the ideal $r\to s$ with probability $1-p$. Then the probability that
one starts at $\mathbbm{1}$ and hits the element $s$ in the ideal in precisely $t$ steps is
\[
	\mathsf{Pr}_G(\tau=t) = (1-p)p^{t-2} \qquad \text{for $t\geqslant 2$.}
\]
Hence
\begin{equation}
\label{equation.Pr example}
	\mathsf{Pr}_G(\tau \geqslant t) = \sum_{j=t}^\infty \mathsf{Pr}_G(\tau=j)
	= (1-p) p^{t-2} \sum_{j=0}^\infty p^j
	= (1-p) p^{t-2} \frac{1}{1-p} = p^{t-2} \quad \text{for $t\geqslant 2$.}
\end{equation}
The expectation value is
\begin{equation}
\label{equation.E example}
	E_G[\tau] = \sum_{t=1}^\infty \mathsf{Pr}_G(\tau\geqslant t) = 1+ \sum_{t=2}^\infty p^{t-2} = 1+ \frac{1}{1-p}.
\end{equation}
Indeed by Markov's inequality
\[
	tp^{t-2} \leqslant 1+\frac{1}{1-p} \quad \text{for all $t\geqslant 2$.}
\]

Now let us use~\eqref{equation.PrG} to compute $\mathsf{Pr}_G(\tau\geqslant t)$. Suppose that the
step $\mathbbm{1}\to r$ is labelled by the generator $1$, the loop from $r$ to $r$ is labelled $2$, and the step
$r\to s$ is labelled $3$. Then the Kleene expression for the paths from $\mathbbm{1}$ to $s$ is
\[
	12^\star 3.
\]
Let the probability for generator $i$ be $x_i$ for $i\in \{1,2,3\}$. Then by~\cite{RhodesSchilling.2019}
\[
	\Psi_G(x_1,x_2,x_3) = \frac{x_1x_3}{1-x_2} = x_1 x_3 \sum_{j=0}^\infty x_2^j.
\]
By~\eqref{equation.PrG}, we obtain $\mathsf{Pr}_G(\tau\geqslant t)=1$ for $t=0,1$ and
\[
	\mathsf{Pr}_G(\tau\geqslant t) = \frac{x_1 x_3 \sum_{j=t-2}^\infty x_2^j}{x_1 x_3 \sum_{j=0}^\infty x_2^j}
	= x_2^{t-2} \qquad \text{for $t\geqslant 2$.}
\]
This agrees with~\eqref{equation.Pr example}, where $x_2=p$.

Next let us use~\eqref{equation.EG} to compute $E_G[\tau]$
\[
	E_G[\tau] = \left(x_1\frac{\partial}{\partial x_1} + x_2\frac{\partial}{\partial x_2} + x_3\frac{\partial}{\partial x_3}\right)
	\ln \Psi_G(x_1,x_2,x_3)
	= 2 + \frac{x_2}{1-x_2} = 1+\frac{1}{1-x_2},
\]
which agrees with~\eqref{equation.E example} when $x_2=p$.
\end{example}

\subsection{Shannon entropy and exponential bounds}
\label{section.shannon}

It turns out that the mixing time has close ties to information theory and in particular Shannon's entropy.
See~\cite{Rioul.2018} and~\cite[Chapter 3]{Gray.2011} as references on information theory.

Let $X$ be a random variable with probability distribution $p(x)$. 
The amount of information of an elementary event $x$ is $\log \frac{1}{p(x)}$.
Therefore, the average amount of information about $X$ is given by the expected value,
known as \defn{Shannon's entropy}
\begin{equation}
	H(X) = E[\log \frac{1}{p}] = \sum_{x \in X} p(x) \log \frac{1}{p(x)}.
\end{equation}

Shannon's entropy features in the \defn{asymptotic equipartition property} or \defn{entropy ergodic theorem}, 
which can be stated as follows~\cite{Shannon.1948} (see also~\cite{Rioul.2018}). Let ${\bf x}=(x_1,\ldots,x_t)$ be a long 
sequence of independent and identically distributed outcomes with probability distribution $p(x)$.
By the independence, $p({\bf x})$ is given by the product
\[
	p({\bf x}) = p(x_1) p(x_2) \cdots p(x_t) = \prod_{x \in X} p(x)^{t(x)},
\]
where $t(x)$ is the number of $x_i$ equal to $x$. Since $t$ is large, by the law of large numbers
\[
	\frac{t(x)}{t} \approx p(x),
\]
which implies
\begin{equation}
	p({\bf x}) \approx \Bigl( \prod_{x \in X} p(x)^{p(x)} \Bigr)^t = e^{-t H(X)}.
\end{equation}
In other words, for very large (but fixed) $t$, the value of the probability of a given ``typical'' sequence 
${\bf x} = (x_1, x_2, \ldots, x_t)$ is likely to be close to the constant $e^{-t H(X)}$.

The precise formulation of the asymptotic equipartition property is the \defn{Shannon--McMillan--Breiman 
Theorem}~\cite{Shannon.1948, McMillam.1953, Breiman.1957} (see 
also~\cite[Chapter 4]{Gray.2011}). Applied to $P^{\star t}(S \setminus K(S))$ in Theorem~\ref{theorem.ASST},
this gives an exponential bound on $\| T^t \nu - \Psi \|$. In probability, this is also known as 
the \defn{Convergence Theorem} (see~\cite[Theorem 4.9]{LevinPeres.2017}).

\begin{theorem}[Convergence Theorem]
Suppose $T$ is the transition matrix of an ergodic Markov chain with stationary distribution $\Psi$.
Then there exist constants $\alpha \in (0,1)$ and $C>0$ such that
\[
	\| T^t \nu - \Psi \| \leqslant C \alpha^t.
\]
\end{theorem}

A concept related to entropy is the \defn{entropy rate}. It is defined as the rate of information innovation
\[
	H' = \lim_{t\to \infty} H(X_t \mid X_{t-1},\ldots, X_1).
\]
When $X_i$ is stationary, the entropy rate is equal to the average entropy per symbol
\[
	\overline{H} = \lim_{t\to \infty} \frac{H(X_1,\ldots,X_t)}{t},
\]
that is $H' = \overline{H}$.

Since an ergodic Markov chain has a unique stationary distribution $\Psi$, the entropy rate is independent of the initial
distribution. If the Markov chain is defined on the finite (or countable) state space $\Omega$, then
\[
	H' = - \sum_{s,s'\in \Omega} T_{s,s'} \Psi_{s'} \log(T_{s,s'}).
\]
A simple consequence of this definition is that indeed a stochastic process with independent and identically distributed 
random variables has an entropy rate that is the same as the entropy of any individual member of the process. 

\subsection{Mixing time via decreasing statistics}
\label{section.decreasing statistics}

In~\cite{ASST.2015,ASST.2015a}, a technique was developed for an upper bound on the mixing time
using a decreasing statistics on the semigroup underlying the Markov chain.

\begin{lemma} \cite[Lemma 3.6]{ASST.2015}
\label{lemma.statisticbound}
Let $\mathcal M$ be an irreducible Markov chain associated to the semigroup $S$ and probability distribution
$0\leqslant p(s) \leqslant 1$ for $s\in S$. We assume that $\{s \in S \mid p(s)>0\}$ generates $S$.
Let $\Psi$ be the stationary distribution and $f\colon S\to \mathbb N$ be a function, 
called a \defn{statistic}, such that:
\begin{enumerate}
\item $f(ss')\leqslant f(s)$ for all $s,s'\in S$;
\item if $f(s)>0$, then there exists $s' \in S$ with $p(s')>0$ such that $f(ss')<f(s)$;
\item $f(s)=0$ if and only if $s \in K(S)$.
\end{enumerate}
Then if $p=\min\{p(s) \mid s \in S, p(s)>0\}$ and $L=f(\mathbbm{1})$, we have that
\[
  \|T^t\nu -\Psi\|_{TV} \leqslant \sum_{i=0}^{L-1} {t\choose i}p^i(1-p)^{t-i}
  \leqslant \exp\left(-\frac{(tp-(L-1))^2}{2tp}\right)\,,
  \]
for any probability distribution $\nu$ on $S$, where the last inequality holds as long as $t\geqslant (L-1)/p$.
\end{lemma}

The bound 
\[
	\sum_{i=0}^{L-1} {t\choose i}p^i(1-p)^{t-i} \leqslant \exp\left(-\frac{(tp-(L-1))^2}{2tp}\right)
\]
works well for $p$ close to $\frac{1}{2}$. A better bound for $0<\frac{L-1}{t}<p$ is given by~\cite{AG.1989}
\[
	\sum_{i=0}^{L-1} {t\choose i}p^i(1-p)^{t-i} \leqslant \exp\left( -t\; D\Bigl(\frac{L-1}{t} \;\Big\|\; p\Bigr) \right),
\]
where
\[
	D(a\; \| \;p) = a \log \frac{a}{p} + (1-a) \log\frac{1-a}{1-p}.
\]
This can be rewritten as
\[
	\sum_{i=0}^{L-1} {t\choose i}p^i(1-p)^{t-i} \leqslant \left( \frac{p}{a} \right)^{ta} \left( \frac{1-p}{1-a} \right)^{t(1-a)},
\]
where $a=\frac{L-1}{t}$.

\subsection{Syntactic at $0$}
\label{section.syntactic}

Syntactic monoids were introduced in mathematics and computer science as the smallest monoid that
recognizes a given formal language, see for example~\cite{Straubing.1994}. Here we develop this idea in the context 
of the mixing time.

Recall that for a semigroup $S$, denote by $S^{\mathbbm{1}}$ the semigroup $S$ with a new added identity 
$\mathbbm{1}$ (even if a one already exists).

\begin{definition}
Let $S$ be a semigroup with zero $0$. Define the congruence on $s_1,s_2 \in S$ by
\begin{equation}
\label{equation.congruence}
	s_1 \equiv s_2 \qquad \text{if and only if} \qquad
	\Bigl( \text{for any $x, y \in S^{\mathbbm{1}}$} \quad  x s_1 y = 0  \Longleftrightarrow x s_2 y =0 \Bigr).
\end{equation}
Then $S$ is called \defn{syntactic at zero} if the congruence~\eqref{equation.congruence} has singleton classes,
that is, 
\[
	S/\equiv \quad \cong \quad S.
\]
We call $S/\equiv$ the \defn{syntactic image} of $S$, which is syntactic at zero. In other words, the syntactic 
semigroup associated to $S$ is the smallest image under the homomorphism $f \colon S \to S/\equiv$ such that 
$f^{-1}(0)=0$.
\end{definition}

\begin{example}
\label{example.min}
Consider the semigroup $S=\{0,1,2,\ldots,n\}$, where multiplication is taking the minimum. The $\equiv$-classes
are given by $\{1,2,\ldots,n\}$ and $\{0\}$. Hence, the syntactic semigroup $S/\equiv$ associated to $S$ is 
isomorphic to $\{0,1\}$ with multiplication being minimum.
\end{example}

\begin{example}
\label{example.rees}
The \defn{Rees matrix semigroup} $(S;I,I';P)$ is indexed by a semigroup $S$, two non-empty sets $I$ and $I'$, and
a matrix $P$ indexed by $I$ and $I'$ with entries $p_{i',i}\in S$ (see for example~\cite[Section 3.4]{RhodesSchilling.2019}).
It is the set $I\times S\times I'$ with multiplication
\[
	(i,s,i')(j,t,j')=(i,sp_{i',j} t,j').	
\]
The \defn{Rees matrix semigroup with zero} $(S;I,I';P)^\square$ is the set  $I\times S\times I' \cup \{\square\}$,
where the entries in $P$ are in $S\cup \{\square\}$, with multiplication
\[
	(i,s,i')(j,t,j')=\begin{cases}
	(i,sp_{i',j} t,j') & \text{if $p_{i',j} \neq \square$,}\\
	\square & \text{otherwise.}
	\end{cases}
\]
Then the syntactic image of $(S;I,I';P)^\square$ is isomorphic to $(\{1\};\tilde{I},\tilde{I}';\tilde{P})$, where
$\tilde{P}$ is a matrix of $0$ and $1$ without equal rows or columns.
\end{example}

It turns out that we can replace a semigroup with zero with its syntactic image without changing the upper bound on
the mixing time of the underlying Markov chain, but the stationary distribution can change.

\begin{theorem}
\label{theorem.syntactic}
Let $(S,A)$ be a finite semigroup $S$ with zero and generators $A$, whose minimal ideal $K(S)$ is a left zero semigroup.
Then the Markov chains associated to $(S,A)$ and $(S/\equiv,f(A))$ have the same upper bound $\mathsf{Pr}(\tau>t)$ on the mixing time.
\end{theorem}

\begin{remark}\mbox{}
\label{remark.syntactic}
\begin{enumerate}
\item
If the probability associated to the generator $a\in A$ is $x_a$, then the probability associated to the generator
$b \in f(A)$ is $\sum_{a \in f^{-1}(b)} x_a$.
\item
Note that the stationary distributions of the Markov chains associated to $(S,A)$ and \newline $(S/\equiv,f(A))$ may differ.
\end{enumerate}
\end{remark}

\begin{proof}[Proof of Theorem~\ref{theorem.syntactic}]
Let $\mathcal{S}$ be the semaphore code corresponding to the ideal $K(S)$. Then for a codeword $s\in \mathcal{S}$,
$f(s)$ is a codeword in the semaphore code corresponding to $K(S/\equiv)$. If the probabilities match up as in
Remark~\ref{remark.syntactic}, the random variable $\tau$ matches and hence the upper bound on the mixing time determined from
$\mathsf{Pr}(\tau>t)$ matches. 
\end{proof}

Theorem~\ref{theorem.syntactic} is powerful in the sense that the upper bound on the mixing time for Markov chains with potentially
complicated stationary distributions can be deduced from those for small semigroups which are syntactic at zero.

\begin{example}
\label{example.semaphore}
Let us continue with Example~\ref{example.min}. The semigroup $(S,A)$ with $S=\{0,1\}$, $A=\{a,b\}$ and $a=0, b=1$ is
syntactic. The minimal ideal $K(S)$ is $A^\star a A^\star$ and the semaphore code is $\mathcal{S}=b^\star a
=\{b^ja \mid j\geqslant 0\}$. The left action on $\mathcal{S}$ is given by
\begin{align*}
	a \cdot b^j a &= a && \text{(reset to $a$),}\\
	b \cdot b^j a &= b^{j+1}a && \text{(free),}
\end{align*}
with stationary distribution
\[
	\Psi_{b^ja} = x_b^j x_a \qquad \text{for $j\geqslant 0$.}
\]
Note that
\[
	E[\tau] = \sum_{j=0}^\infty (j+1) x_b^j x_a 
	= x_a \frac{\partial}{\partial x_b} \left( \sum_{j=0}^\infty x_b^{j+1} \right)
	= x_a \frac{\partial}{\partial x_b} \frac{x_b}{1-x_b}
	= \frac{x_a}{(1-x_b)^2} = \frac{1}{x_a}.
\]
In contrast, let us compute
\[
	\mathsf{Pr}(\tau>t) = \sum_{j=t}^\infty x_b^j x_a
	= \frac{x_a x_b^{t}}{1-x_b} = x_b^{t}.
\]
Indeed $\mathsf{Pr}(\tau>t) \leqslant \frac{E[\tau]}{t+1}$ as in Example~\ref{example.single loop}.
\end{example}

\begin{example}
We can amend Example~\ref{example.semaphore} by making the semigroup finite and aperiodic by imposing 
$b^w = b^{w+1}$. Using the methods in~\cite{RhodesSchilling.2019} (or comparing the in-flow with the
out-flow), the stationary distribution can be derived to be
\[
\begin{split}
	\Psi_{b^ja} &= x_b^j x_a \qquad \text{for $0\leqslant j < w$,}\\
	\Psi_{b^wa} &= \frac{x_a x_b^w}{1-x_b}.
\end{split}
\]
The associated syntactic semigroup is $(\{0,1\},A)$, which means by Theorem~\ref{theorem.syntactic} that the upper bound on the mixing 
time is unchanged, even though the stationary distribution is different.
\end{example}

\begin{example}
Let $(S,A)$ be an arbitrary finite semigroup with generators $A=\{a_1,\ldots,a_k\}$ (with or without zero). Let $S^\square$
be the semigroup with a zero $\square$ adjoined. Then
\[
	\left( S^\square / \equiv \right) = (\{ \square,1\},A\cup \{ \square \}).
\]
In this setting the stationary distribution can be complicated, however the upper bound on the mixing time is trivial by 
Theorem~\ref{theorem.syntactic}
\[
	\mathsf{Pr}(\tau>t) = (1-x_\square)^t.
\]	
\end{example}

\begin{example}
Consider the Rees matrix semigroup $S=B(2)$ of~\cite[Example 3.3]{RhodesSchilling.2019} with generators $A=\{a,b\}$, where
$a=(1,2)$ and $b=(2,1)$. The minimal ideal $K(S)$ is $A^\star \{aa,bb\} A^\star$ with semaphore code
\[
	\mathcal{S} = \{(ab)^\star aa, (ba)^\star bb, b(ab)^\star aa, a(ba)^\star bb\}.
\]
The left action on $\mathcal{S}$ is given by
\begin{align*}
	a &\cdot (ab)^j aa = aa && \text{(reset),}\\
	a &\cdot (ba)^j bb = a(ba)^j bb && \text{(free)},\\
	a &\cdot b(ab)^j aa = (ab)^{j+1} aa && \text{(free)},\\
	a &\cdot a(ba)^j bb = aa && \text{(reset),}
\end{align*}
and similarly with $a$ and $b$ interchanged. Note that
\begin{align*}
	\mathsf{Pr}(\tau>2k) &= \sum_{j=k}^\infty (x_a^2 + x_b^2 + x_a  + x_b) (x_ax_b)^j
	= \frac{(x_a x_b)^k(x_a^2 + x_b^2 + 1)}{1-x_a x_b} = 2(x_a x_b)^k,\\
	\mathsf{Pr}(\tau>2k+1) &= \sum_{j=k}^\infty (x_a^2 + x_b^2 + x_a^2 x_b + x_b^2 x_a) (x_ax_b)^j
	= \frac{(x_a x_b)^k(x_a^2 + x_b^2 + x_a x_b)}{1-x_a x_b} = (x_a x_b)^k,
\end{align*}
which by Theorem~\ref{theorem.ASST} gives an upper bound on the mixing time.
\end{example}

\begin{example}
\label{example.ress aa}
Consider the Rees matrix semigroup (see Example~\ref{example.rees}) with $I=I'=\{1,2\}$, $S=\{0,1\}$,
\[
	P = \begin{pmatrix} 1&1\\ 0&1 \end{pmatrix},
\]
and generators $A=\{a,b\}$ with $a=(1,1,2)$ and $b=(2,1,1)$. The minimal ideal is $A^\star aa A^\star$ with 
semaphore code $\mathcal{S} = b^\star (abb^\star)^\star aa$. The left action on $\mathcal{S}$ is given by
\[
\begin{split}
	a \cdot b^j \left(\prod_{k=1}^\ell abb^{e_k} \right) aa &= \begin{cases} ab^j \left(\prod_{k=1}^\ell abb^{e_k} \right) aa
	& \text{if $j>0$ (free),} \\
	aa & \text{if $j=0$ (reset),} \end{cases}\\
	b \cdot b^j \left(\prod_{k=1}^\ell abb^{e_k} \right) aa &= b^{j+1} \left(\prod_{k=1}^\ell abb^{e_k} \right) aa
	\qquad \text{(free).}
\end{split}
\]
In this case, the bound on the mixing time is given by
\[
	\mathsf{Pr}(\tau>k) = x_a^2 \sum_{j \geqslant k-1} \sum_{i=0}^{\lfloor \frac{j}{2} \rfloor} \binom{j-i}{i} x_a^i x_b^{j-i}.
\]
\end{example}

\subsection{Ideals and $d$-testable languages}
\label{section.languages}

As we have seen, ideals are important in the study of Markov chains in the context of semigroups. In addition, ideals are 
closely related to semaphore codes.

Let $(S_j,A)$ be two semigroups with zero for $j=1,2$ with the same generating set $A$ and $I_j$ the ideal of strings
in $A^+$ that is zero in $(S_j,A)$ for $j=1,2$. Let $\mathcal{S}_j$ for $j=1,2$ be the semaphore code associated to the 
ideal $I_j$. Recall that through the left action of $A^+$ on $\mathcal{S}_j$ we have two Markov chains.

\begin{remark}[\defn{Ideal principle}]
\label{remark.ideal containment}
If $I_1 \subseteq I_2$, the upper bound on the mixing time of the Markov chain associated to $\mathcal{S}_2$ is smaller or equal to the
upper bound on the mixing time of the Markov chain associated to $\mathcal{S}_1$.
\end{remark}

Remark~\ref{remark.ideal containment} is true since by~\cite[Corollary 3.5(3)]{ASST.2015} the mixing time is bounded 
above by $P^{\star t}(S \setminus K(S))$ (see Theorem~\ref{theorem.ASST}). If $I_1\subseteq I_2$, we hence have
\[
	P^{\star t}(S_2 \setminus I_2) \leqslant P^{\star t}(S_1 \setminus I_1),
\]
since $I_j$ consists of all words in $A^+$ which are zero in $S_j$.

By Remark~\ref{remark.ideal containment} we want to study Markov chains with the smallest ideals as they have the
worst mixing time. To this end, we will study the \defn{complete lattice of ideals} of $A^\star$. All ideals (including
$\emptyset$) of $A^\star$ form a complete lattice under union and intersection.

\begin{lemma}
\label{lemma.descending chain}
Every nonempty ideal $I$ has a \defn{descending chain}
\[
	I \supset A^\star t_1 A^\star \supset A^\star t_2 A^\star \supset \cdots \supset A^\star t_k A^\star \supset
	\cdots.
\]
\end{lemma}

\begin{proof}
Since $I\neq \emptyset$, there exists an element $t_1\in I$. The unique smallest length element in $A^\star t_1 A^\star$
is of length $|t_1|$. Choose $t_2 \in A^\star t_1 A^\star$ with $|t_2|>|t_1|$. Then $A^\star t_1 A^\star \supset 
A^\star t_2 A^\star$ and repeat.
\end{proof}

Some ideals $I$ have an infinite \defn{ascending chain}
\[
	I \subset I_1 \subset I_2 \subset \cdots
\]
and some do not. Let $A=\{a,b\}$. Then $A^\star \setminus \{a\}$, for example, does not have an infinite ascending chain.
On the other hand (compare also with Example~\ref{example.ress aa})
\[
	A^\star aa A^\star \subset A^\star aa A^\star \cup A^\star aba A^\star \subset \cdots \subset
	\bigcup_{j=0}^k A^\star ab^j a A^\star \subset \cdots
\]
does.

Every ideal $I \subseteq A^+$ has a unique set of minimal generators, namely all $t=a_1 a_2 \cdots a_{\ell-1} a_\ell \in I$ 
such that $a_1\cdots a_{\ell-1} \not \in I$ and $a_2 \cdots a_\ell \not \in I$. Hence by Lemma~\ref{lemma.descending chain}, 
the smallest ideals are of the form $A^\star t A^\star$, where $|t|$ is big. Since by Remark~\ref{remark.ideal containment} 
smaller ideals have worse upper bounds on the mixing times, we would like to analyze ideals of the form $A^\star t A^\star$, where $|t|$ is large.
This is related to $d$-testable languages, which are finite ideals generated by $\bigcup_{i=1}^n A^\star t_i A^\star$, 
see~\cite{Zalcstein.1972}.

Let $t\in A^+$. The minimal automata $\mathsf{Test}(t)$ accepting the language $A^\star t A^\star$ for 
$t=a_1a_2\ldots a_\ell$ is given as follows. There are $\ell+1$ states: $\mathbbm{1}, a_1, a_1 a_2,\ldots, a_1 \ldots 
a_{\ell-1}, a_1\ldots a_\ell \equiv 0$. 
We have $q \stackrel{a}{\longrightarrow} qa$ if both $q$ and $qa$ are prefixes of $t$ and otherwise 
$q \stackrel{a}{\longrightarrow} \mathbbm{1}$.

Using~\cite[Definition 3.5]{RhodesSchilling.2019a}, $\mathsf{Test}(t)$ can be transformed into a loop graph with loops
labeled by words $w\in A^+$ such that $|w| \leqslant \ell$, $w=w_1\cdots w_k$ is not a prefix of $t$, but $w_1\cdots
w_{k-1}$ is a prefix of $t$. Let us denote the set of all such words $W_t$.
Hence the Kleene expression for the paths in $\mathsf{Test}(t)$ is $\left(\cup_{w \in W_t} \{w\}\right)^\star t$ and hence
the stationary distribution is
\[
	\Psi_t = \frac{x_{a_1} \cdots x_{a_\ell}}{1-\sum_{w\in W_t} \prod_{a\in w} x_a}.
\]
By Theorem~\ref{theorem.main1} we obtain
\[
	E_t[\tau] = \ell + \frac{\sum_{w\in W_t} |w| \prod_{a\in w} x_a}{1-\sum_{w\in W_t} \prod_{a\in w} x_a},
\]
which gives an upper bound on the mixing time using the Markov inequality~\eqref{equation.Markov inequality}.
Theorem~\ref{theorem.main} can also be used to obtain an upper bound on the mixing time using the series expansion of 
$\Psi_t$.

\begin{remark}
The loop graphs in~\cite{RhodesSchilling.2019a} are not allowed to have loops at vertex $\mathbbm{1}$. Here we do
allow loops at $\mathbbm{1}$. To remedy the situation, one could rename $\mathbbm{1}$ by $1$ and have an edge
with probability 1 from $\mathbbm{1}$ to $1$.
\end{remark}

\begin{example}
Let $A=\{a,b\}$ and $t=aba$. Then $\mathsf{Test}(t)$ can be depicted by
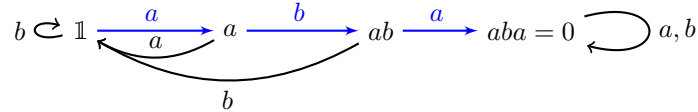
\begin{center}
\begin{tikzpicture}[auto]
\node (I) at (0, 0) {$\mathbbm{1}$};
\node (a) at (2,0) {$a$};
\node (ab) at (4,0) {$ab$};
\node (aba) at (6,0) {$aba=0$};
\draw[edge,blue,thick] (I) -- (a) node[midway, above] {$a$\;};
\draw[edge,blue,thick] (a) -- (ab) node[midway, above] {$b$\;};
\draw[edge,blue,thick] (ab) -- (aba) node[midway, above] {$a$\;};
\path (a) edge[->,thick, bend left=30] node[midway,above] {$a$} (I);
\path (ab) edge[->,thick, bend left=30] node[midway,below] {$b$} (I);
\path
	(I) edge [loop left,thick] node {$b$} (I)
	(aba) edge [loop right, thick] node {$a,b$} (aba);
\end{tikzpicture}
\end{center}
The corresponding loop graph is
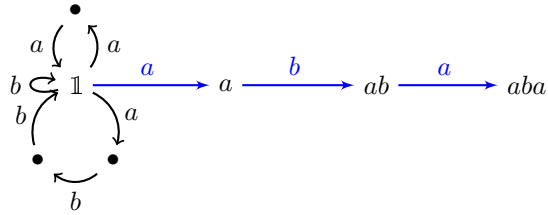
\begin{center}
\begin{tikzpicture}[auto]
\node (I) at (0, 0) {$\mathbbm{1}$};
\node (a) at (2,0) {$a$};
\node (ab) at (4,0) {$ab$};
\node (aba) at (6,0) {$aba$};
\node (dot) at (0,1) {$\bullet$};
\node (dot1) at (0.5,-1) {$\bullet$};
\node (dot2) at (-0.5,-1) {$\bullet$};
\draw[edge,blue,thick] (I) -- (a) node[midway, above] {$a$\;};
\draw[edge,blue,thick] (a) -- (ab) node[midway, above] {$b$\;};
\draw[edge,blue,thick] (ab) -- (aba) node[midway, above] {$a$\;};
\path
	(I) edge [loop left,thick] node {$b$} (I)
	(I) edge[->, thick, bend right=40] node[midway,right]{$a$} (dot)
	(dot) edge[->, thick, bend right=40] node[midway,left]{$a$} (I)
	(I) edge[->, thick, bend left=40] node[midway,right]{$a$} (dot1)
	(dot1) edge[->, thick, bend left=40] node[midway,below]{$b$} (dot2)
	(dot2) edge[->, thick, bend left=40] node[midway,left]{$b$} (I);
\end{tikzpicture}
\end{center}
Hence the stationary distribution is
\[
	\Psi_{aba} = \frac{x_a^2 x_b}{1-x_b-x_a^2-x_ax_b^2}.
\]
By Theorem~\ref{theorem.main1}, this hence gives
\[
	E_t[\tau] = 3 + \frac{x_b + 2 x_a^2 + 3 x_a x_b^2}{1-x_b-x_a^2-x_a x_b^2}.
\]
\end{example}

\begin{example}
Now let us take $A=\{a,b\}$ and $t=a^\ell$. In this case $W_t = \{a^kb \mid 0\leqslant k<\ell\}$ and hence
\[
	\Psi_t = \frac{x_a^\ell}{1-\sum_{k=0}^{\ell-1} x_a^k x_b}
\]
with an upper bound for the mixing time given by
\[
	E_t[\tau] = \ell + \frac{\sum_{k=0}^{\ell-1} (k+1) x_a^k x_b}{1-\sum_{k=0}^{\ell-1} x_a^k x_b}
\]
using~\eqref{equation.Markov inequality}.
\end{example}

\section{Examples}
\label{section.examples}

In this section, we analyze the mixing time of several examples using the methods developed in Section~\ref{section.mixing time}.
In Section~\ref{section.tsetlin} we derive upper bounds for the mixing time of the famous Tsetlin library~\cite{Tsetlin.1963}
and in Section~\ref{section.edge flipping} for edge flipping on a line~\cite{ChungGraham.2012}. In Section~\ref{section.promotion},
we provide a new Markov chain on linear extension of a poset with $n$ vertices, inspired by but different from the promotion Markov 
chain of Ayyer, Klee and the last author. The mixing time of this Markov chain is $O(n \log n)$ (Theorem~\ref{theorem.expected new}).

\subsection{The Tsetlin library}
\label{section.tsetlin}

The Tsetlin library~\cite{Tsetlin.1963} is a Markov chain whose states are all permutations $S_n$ of $n$ books (on a shelf).
Given $\pi \in S_n$, construct $\pi' \in S_n$ from $\pi$ by removing book $a$ from the shelf and inserting 
it to the front. In this case write $\pi \stackrel{a}{\longrightarrow} \pi'$. 
Let $0< x_a\leqslant 1$ be probabilities for each $1\leqslant a \leqslant n$ such that $\sum_{a=1}^n x_a = 1$. 
In the Tsetlin library Markov chain, we transition $\pi \stackrel{a}{\longrightarrow} \pi'$ with probability $x_a$. 
The stationary distribution for the Tsetlin library was derived by Hendricks~\cite{Hendricks.1972, Hendricks.1973}
and Fill~\cite{Fill.1996}
\begin{equation}
\label{equation.psi Tsetlin}
	\Psi_\pi = \prod_{i=1}^n \frac{x_{\pi_i}}{1-\sum_{j=1}^{i-1} x_{\pi_j}} \qquad \text{for all $\pi \in S_n$.}
\end{equation}
The stationary distribution was derived using right Cayley graphs and their Karnofsky--Rhodes and McCammond
expansions in~\cite[Section 3.1]{RhodesSchilling.2019}.

Consider the semigroup $P(n)$, which consists of the set of all non-empty subsets of $\{1,2,\ldots,n\}$. Multiplication in 
$P(n)$ is union of sets. We pick as generators $A=[n]:=\{1,2,\ldots,n\}$. Then the right Cayley graph $\mathsf{RCay}(P(n),[n])$ 
is the Boolean poset with $\mathbbm{1}$ as root. The right Cayley graph for $P(3)$ is depicted in Figure~\ref{figure.P3}.
Except for the loops at a given vertex, all edges are transitional. Hence $\mathsf{Mc} \circ \mathsf{KR}(P(n),[n]) 
= \mathsf{KR}(P(n),[n])$ is a tree with leaves given by the permutations $S_n$ of $[n]$.
The case $n=3$ is depicted in Figure~\ref{figure.Mc P3}.

\begin{figure}[t]
\begin{tikzpicture}[auto]
\node (I) at (0, 0) {$\mathbbm{1}$};
\node (A) at (-3,-1.5) {$\{1\}$};
\node(B) at (0,-1.5) {$\{2\}$};
\node(C) at (3,-1.5) {$\{3\}$};
\node(D) at (-3,-3) {$\{1,2\}$};
\node(E) at (0,-3) {$\{1,3\}$};
\node(F) at (3,-3) {$\{2,3\}$};
\node(G) at (0,-4.5) {$\{1,2,3\}$};

\draw[edge,blue,thick] (I) -- (A) node[midway, left] {$1$\;};
\draw[edge,blue,thick] (I) -- (B) node[midway, right] {$2$};
\draw[edge,blue,thick] (I) -- (C) node[midway, right] {\;$3$};
\draw[edge,blue,thick] (A) -- (D) node[midway,left] {$2$};
\draw[edge,blue,thick] (A) -- (E) node[midway,above] {$3$};
\draw[edge,blue,thick] (B) -- (D) node[midway,below] {$1$};
\draw[edge,blue,thick] (C) -- (F) node[midway,right] {$2$};
\draw[edge,blue,thick] (B) -- (F) node[midway,below] {$3$};
\draw[edge,blue,thick] (C) -- (E) node[midway,above] {$1$};
\draw[edge,blue,thick] (D) -- (G) node[midway,below] {$3$};
\draw[edge,blue,thick] (E) -- (G) node[midway,right] {$2$};
\draw[edge,blue,thick] (F) -- (G) node[midway,right] {\;$1$};
\path
	(A) edge [loop left] node {$1$} (A)
	(B) edge [loop right] node {$2$} (B)
	(C) edge [loop right] node {$3$} (C)
	(D) edge [loop left] node {$1,2$} (D)
	(E) edge [loop right] node {$1,3$} (E)
	(F) edge [loop right] node {$2,3$} (F)
	(G) edge [loop right] node {$1,2,3$} (G);
\end{tikzpicture}
\caption{The right Cayley graph $\mathsf{RCay}(S,A)$ with $S=P(3)$ and $A=\{1,2,3\}$. Transition edges are drawn in blue.
\label{figure.P3}}
\end{figure}
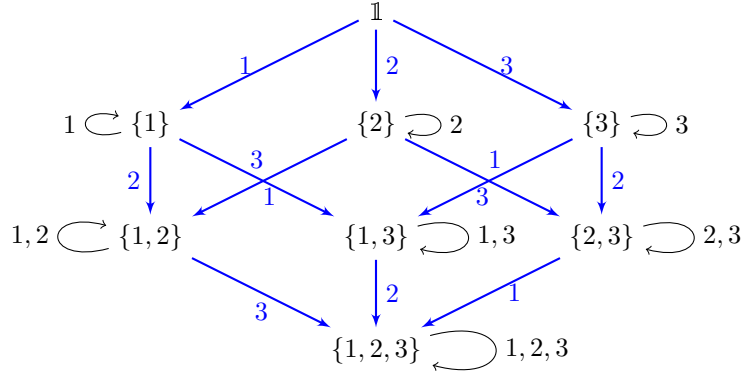

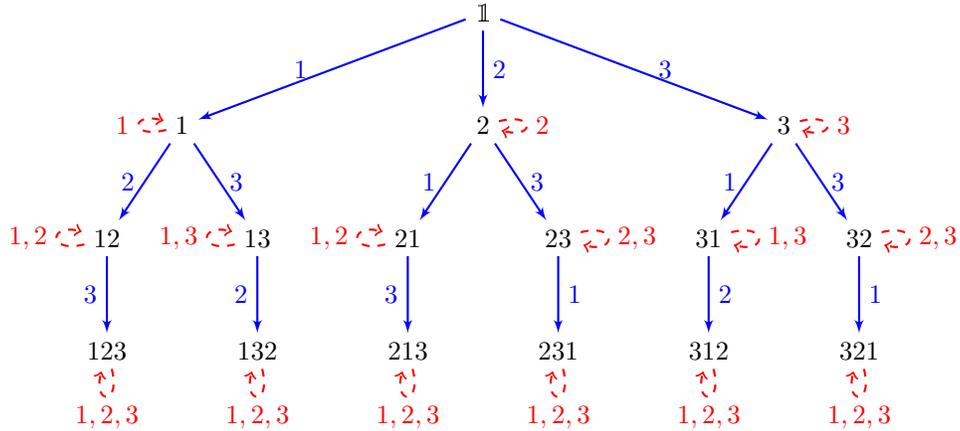
\begin{figure}[t]
\begin{tikzpicture}[auto]
\node (I) at (0, 0) {$\mathbbm{1}$};
\node (A) at (-4,-1.5) {$1$};
\node(B) at (0,-1.5) {$2$};
\node(C) at (4,-1.5) {$3$};
\node(D) at (-5,-3) {$12$};
\node(E) at (-3,-3) {$13$};
\node(F) at (-1,-3) {$21$};
\node(G) at (1,-3) {$23$};
\node(H) at (3,-3) {$31$};
\node(J) at (5,-3) {$32$};
\node(DD) at (-5,-4.5) {$123$};
\node(EE) at (-3,-4.5) {$132$};
\node(FF) at (-1,-4.5) {$213$};
\node(GG) at (1,-4.5) {$231$};
\node(HH) at (3,-4.5) {$312$};
\node(JJ) at (5,-4.5) {$321$};

\draw[edge,blue,thick] (I) -- (A) node[midway, left] {$1$\;\;};
\draw[edge,blue,thick] (I) -- (B) node[midway, right] {$2$};
\draw[edge,blue,thick] (I) -- (C) node[midway, right] {\;\;$3$};
\draw[edge,blue,thick] (A) -- (D) node[midway, left] {$2$};
\draw[edge,blue,thick] (A) -- (E) node[midway, right] {$3$};
\draw[edge,blue,thick] (B) -- (F) node[midway, left] {$1$};
\draw[edge,blue,thick] (B) -- (G) node[midway, right] {$3$};
\draw[edge,blue,thick] (C) -- (H) node[midway, left] {$1$};
\draw[edge,blue,thick] (C) -- (J) node[midway, right] {$3$};
\draw[edge,blue,thick] (D) -- (DD) node[midway, left] {$3$};
\draw[edge,blue,thick] (E) -- (EE) node[midway, left] {$2$};
\draw[edge,blue,thick] (F) -- (FF) node[midway, left] {$3$};
\draw[edge,blue,thick] (G) -- (GG) node[midway, right] {$1$};
\draw[edge,blue,thick] (H) -- (HH) node[midway, right] {$2$};
\draw[edge,blue,thick] (J) -- (JJ) node[midway, right] {$1$};

\path
	(A) edge [loop left, red, dashed,thick] node {$1$} (A)
	(B) edge [loop right, red, dashed, thick] node {$2$} (B)
	(C) edge [loop right, red, dashed,thick] node {$3$} (C)
	(D) edge [loop left, red, dashed,thick] node {$1,2$} (D)
	(E) edge [loop left, red, dashed,thick] node {$1,3$} (E)
	(F) edge [loop left, red, dashed,thick] node {$1,2$} (F)
	(G) edge [loop right, red, dashed,thick] node {$2,3$} (G)
	(H) edge [loop right, red, dashed,thick] node {$1,3$} (H)
	(J) edge [loop right, red, dashed,thick] node {$2,3$} (J)
	(DD) edge [loop below, red, dashed,thick] node {$1,2,3$} (DD)
	(EE) edge [loop below, red, dashed,thick] node {$1,2,3$} (EE)
	(FF) edge [loop below, red, dashed,thick] node {$1,2,3$} (FF)
	(GG) edge [loop below, red, dashed,thick] node {$1,2,3$} (GG)
	(HH) edge [loop below, red, dashed,thick] node {$1,2,3$} (HH)
	(JJ) edge [loop below, red, dashed,thick] node {$1,2,3$} (JJ);
\end{tikzpicture}
\caption{$\mathsf{Mc} \circ \mathsf{KR}(P(3),[3]) = \mathsf{KR}(P(3),[3])$, which is the 
Karnofsky--Rhodes expansion of the right Cayley graph of Figure~\ref{figure.P3}.
\label{figure.Mc P3}}
\end{figure}

To obtain an upper bound on the mixing time, we compute $E[\tau]$ from the Karnofsky--Rhodes expansion
of the right Cayley graph. The ideal consists of the leaves of the tree $\mathsf{KR}(P(n),[n])$, which are
labeled by permutations in $S_n$. Recall that $E[\tau]$ can be computed via~\eqref{equation.Etau}.
Any path from $\mathbbm{1}$ to the ideal is of length at least $n$. Hence $\mathsf{Pr}(\tau\geqslant t)=1$
for $1\leqslant t\leqslant n$.

Now for concreteness consider the loop graph $G$ associated to the path from $\mathbbm{1}$ to $12\ldots n$
in $\mathsf{Mc}\circ \mathsf{KR}(P(n),[n])$. The contributions of the loops can be treated in a similar fashion to 
Example~\ref{example.single loop}. The Kleene expression for all paths from $\mathbbm{1}$ to $12\ldots n$ is given by 
\[
	1 1^\star 2 \{1,2\}^\star 3 \{1,2,3\}^\star \ldots \{1,2,\ldots,n-1\}^\star n.
\]
Hence we obtain (compare with~\eqref{equation.psi Tsetlin})
\[
	\Psi_G(x_1,\ldots,x_n) = \frac{x_1 \cdots x_n}{(1-x_1)(1-x_1-x_2) \cdots (1-x_1-\cdots-x_{n-1})}
\]
and by Theorem~\ref{theorem.main1} 
\begin{equation}
\label{equation.EG Tsetlin}
	E_G[\tau]=n + \frac{x_1}{1-x_1} + \frac{x_1+x_2}{1-x_1-x_2} + \cdots + \frac{x_1+\cdots + x_{n-1}}{1-x_1-\cdots-x_{n-1}},
\end{equation}
which can also be checked directly. If $x_i=\frac{1}{n}$ for all $1\leqslant i\leqslant n$, we hence have
\begin{equation}
\label{equation.EG Tsetlin n}
	E_G[\tau] = n + \frac{1}{n-1} + \frac{2}{n-2}+\cdots + \frac{n-1}{1} = n \left( \sum_{i=1}^n \frac{1}{i} \right).
\end{equation}
The last equality can be proved by induction on $n$. It is well-known that the sequence
$t_n = \sum_{i=1}^n \frac{1}{i} - \ln(n)$ approaches the Euler--Mascheroni constant $\gamma$ as $n\to \infty$.
Therefore
\[
	E[\tau]=E_G[\tau] \leqslant n \ln(n) +n \gamma
\]
and by~\eqref{equation.Markov inequality}
\[
	\| T^t \nu - \pi \| \leqslant \frac{n \ln(n) +n \gamma}{t+1}.
\]

Nestoridi~\cite{Nestoridi.2019} has proven upper/lower bounds for the mixing time of the separation distance. Pike~\cite{Pike.2013}
has discussed the eigenfunctions of the transition matrix. Note that, given the rational expression of the stationary 
distribution~\eqref{equation.psi Tsetlin}, our methods work for general weights $x_i$. Truncating the degree of the expansion
of the stationary distribution~\eqref{equation.psi Tsetlin} gives a precise expression for an upper bound of the mixing time by
Theorem~\ref{theorem.main}.

\subsection{Edge flipping on a line}
\label{section.edge flipping}

In~\cite[Section 3.2]{RhodesSchilling.2019}, we treated the Markov chain obtained by edge flipping on a line
using the semigroup methods of~\cite{RhodesSchilling.2019}. Take a line with $n+1$ vertices.
Each vertex can either be $0$ or $1$. So the state space is $\Omega=\{0,1\}^{n+1}$ of size $2^{n+1}$. 
Pick edge $i$ for $1\leqslant i \leqslant n$ (between vertices $i$ and $i+1$) with probability $x_i$. Then with 
probability $\frac{1}{2}$ make the adjacent vertices both 0 (respectively both 1). Let us call this Markov chain
$\mathcal{M}$. This Markov chain is a Boolean arrangement~\cite{BHR.1999} for which the stationary distribution was 
derived in~\cite{BrownDiaconis.1998} and which was also analyzed in~\cite{ChungGraham.2012}.

In~\cite[Section 3.2]{RhodesSchilling.2019}, we analyzed the stationary distribution in a similar fashion to the
Tsetlin library by considering the semigroup $P^{\pm}(n)$, which is the set of signed subsets of $[n]$. That is, take a 
subset of $[n]$ and in addition associate to each letter a sign $+$ or $-$. Right multiplication of such a subset $X$ by 
a generator $x\in [\pm n] := \{\pm 1, \ldots, \pm n\}$ is addition of $x$ to $X$ if neither $x$ nor $-x$ are in $X$ and 
otherwise return $X$. The minimal ideal in the Karnofsky--Rhodes expansion of this monoid is the set of signed 
permutations $S_n^\pm$. In the Markov chain on the minimal ideal, we transition from $\pi \stackrel{a}{\longrightarrow} \pi'$ 
with probability $y_a$ for $a\in [\pm n]$, where $\pi'$ is obtained from $\pi$ by prepending $a$ to $\pi$ and removing the 
letter $a$ or $-a$ from $\pi$. The stationary distribution associated to $\pi \in S^\pm_n$ was computed to be
\begin{equation}
\label{equation.psi signed}
	\Psi^{\mathsf{KR}(P^\pm(n),[\pm n])}_\pi = \prod_{i=1}^n \frac{y_{\pi_i}}{1-\sum_{j=1}^{i-1} (y_{\pi_j} + y_{-\pi_j})}.
\end{equation}
The stationary distribution for a word $s \in \Omega$ for the Markov chain $\mathcal{M}$ is a lumping (or sum)
of the $\Psi^{\mathsf{KR}(P^\pm(n),[\pm n])}_\pi$ in~\eqref{equation.psi signed}. By the same analysis as in
Section~\ref{section.tsetlin} the mixing time for $\mathcal{M}$ is of order $O(n \ln(n))$.

\subsection{Promotion Markov chain}
\label{section.promotion}

Let $P$ be a \defn{partially ordered set}, also known as a \defn{poset}, on $n$ elements with partial order $\preccurlyeq$. 
A partial order must be reflexive ($a \preccurlyeq a$ for all $a\in P$), antisymmetric ($a\preccurlyeq b$ and $b \preccurlyeq a$
implies $a=b$ for $a,b\in P$), and transitive ($a\preccurlyeq b$ and $b\preccurlyeq c$ implies $a\preccurlyeq c$
for $a,b,c\in P$). We assume that the elements of $P$ are labeled by integers in $[n]:=\{1,2,\ldots,n\}$ such that
if $i,j \in P$ with $i \preccurlyeq j$ then $i \leqslant j$ as integers.
Let $\mathcal{L}:=\mathcal{L}(P)$ be the set of \defn{linear extensions} of $P$ defined as
\[
	\mathcal{L}(P) = \{ \pi \in S_{n} \mid i \prec j \text{ in $P$ } \implies \pi^{-1}_{i} < \pi^{-1}_{j} \text{ as integers} \}.
\]

In computer science, linear extensions are also known as \defn{topological sortings} \cite{Knuth-volume1.1997,
Knuth-volume3.1998}. Computing the number of linear extensions is an important problem for
real world applications~\cite{KarzanovKhachiyan.1991}. For example, it relates to sorting algorithms.
Suppose one wants to schedule a sequence of tasks based on their dependencies. Specifying that
a certain task has to come before another task gives rise to a partial order. A linear extension gives a total order 
in which to perform the jobs. 
In social sciences, linear extensions are used in voting procedures~\cite{FishburnGehrlein.1975,Ackermanetal.2013},
where voters rank the candidates according specified traits (view on foreign policies, view on domestic policies etc).
A recursive formula for the number of linear extensions for a given poset $P$ was given in~\cite{EHS.1989}.
Brightwell and Winkler~\cite{BrightwellWinkler.1991} showed that counting the number
of linear extensions is $\# P$-complete. Bubley and Dyer~\cite{BubleyDyer.1999} provided an algorithm to (almost)
uniformly sample the set of linear extensions of a finite poset of size $n$ with mixing time $O(n^3 \log n)$.
In~\cite{AyyerKleeSchilling.2014}, the promotion Markov chain was introduced, which is a random walk on the
linear extensions of a finite poset $P$. Here we discuss a variant of the promotion Markov chain which
has mixing time of order $O(n \log n)$.

\subsubsection{The model}
We now explain the promotion Markov chain introduced in~\cite{AyyerKleeSchilling.2014}.
For a given poset $P$ with $n$ vertices, the state space of the \defn{promotion Markov chain} is the set of
linear extensions $\mathcal{L}(P)$. For $\pi,\pi'\in \mathcal{L}(P)$, we transition 
$\pi \stackrel{\partial_j}{\longrightarrow} \pi'$ with probability $x_{\pi_j}$ if $\pi'=\partial_j \pi$, where $\partial_j$
is the promotion operator. The promotion operator is defined in terms of more elementary operators $\tau_i$ 
($1\leqslant i<n$) which appeared in~\cite{Haiman.1992, MalvenutoReutenauer.1994, Stanley.2009} and
was used explicitly to count linear extensions in~\cite{EHS.1989}. 
Let $\pi=\pi_1 \ldots \pi_n \in \mathcal{L}(P)$ be a linear extension of $P$ in one-line notation. Then
\begin{equation} 
\label{equation.tau}
	\tau_i \pi = \begin{cases}
	\pi_1 \ldots \pi_{i-1} \pi_{i+1} \pi_i \ldots \pi_n & \text{if $\pi_i$ and $\pi_{i+1}$ are not comparable in $P$,}\\
	\pi_1 \ldots \pi_n & \text{otherwise.} \end{cases}
\end{equation}
In other words, $\tau_i$ acts non-trivially on a linear extension if interchanging entries $\pi_i$ and $\pi_{i+1}$ yields another 
linear extension. Then the \defn{promotion operator} on $\mathcal{L}(P)$ is defined as
\begin{equation}
\label{equation.promotion tau}
 	\partial_j = \tau_1 \tau_2 \cdots \tau_{j-1}.
\end{equation}
Note that we use a different convention here to~\cite{AyyerKleeSchilling.2014}, where $\partial_j = \tau_j \tau_{j+1} 
\cdots \tau_{n-1}$. Our convention here is compatible with the conventions for the Tsetlin library as in 
Section~\ref{section.tsetlin}, where we moved letters to the front of the word rather than the end of the word.

\begin{example}
\label{example.promotion}
Let $P$ be the poset on four vertices defined by its covering relations $\{ (1,4), (2,4), (2,3) \}$. Then its Hasse diagram
is the following:
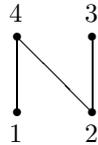
\setlength{\unitlength}{1mm}
\begin{center}
\begin{picture}(20, 20)
\put(10,4){\circle*{1}}
\put(20,4){\circle*{1}}
\put(9,0){1}
\put(19,0){2}
\put(10,14){\circle*{1}}
\put(20,14){\circle*{1}}
\put(9,16){4}
\put(19,16){3}
\put(10,4){\line(0,1){10}}
\put(20,4){\line(0,1){10}}
\put(10,14){\line(1,-1){10}}
\end{picture}
\end{center}
This poset has five linear extensions
\begin{equation}
\label{equation.LP example}
	\mathcal{L}(P) = \{ 1234, 1243, 2134,  2143, 2314 \}.
\end{equation}
The promotion Markov chain for $P$ is depicted in Figure~\ref{figure.promotion}, where the vertices are the
linear extensions and an arrow labelled by $i$ from $\pi$ to $\pi'$ indicates that $\pi' = \partial_i \pi$.

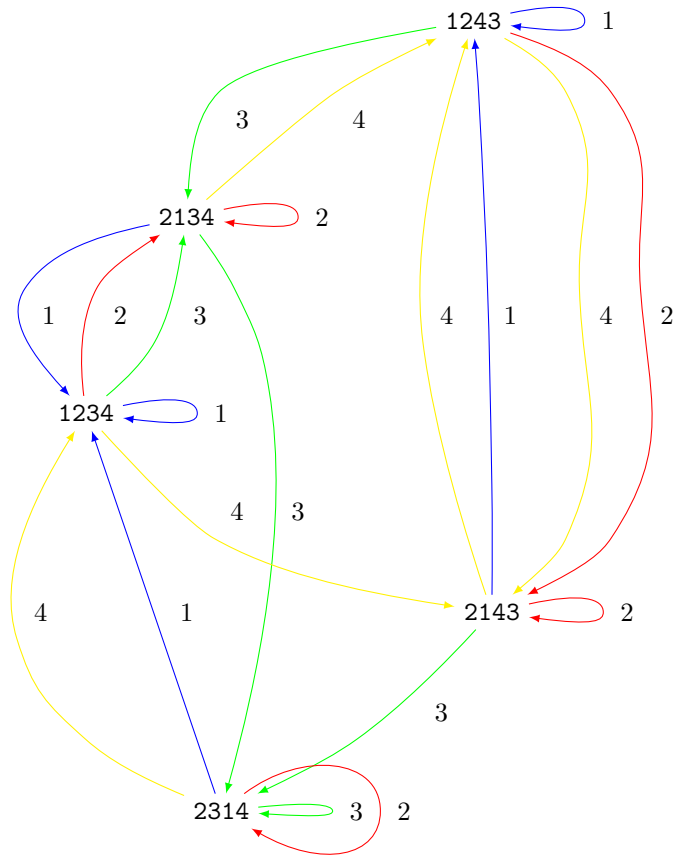
\begin{figure}
\begin{tikzpicture}[>=latex,line join=bevel,]
\node (node_4) at (181.42bp,92.892bp) [draw,draw=none] {$\mathtt{2143}$};
  \node (node_3) at (66.421bp,241.89bp) [draw,draw=none] {$\mathtt{2134}$};
  \node (node_2) at (28.421bp,167.89bp) [draw,draw=none] {$\mathtt{1234}$};
  \node (node_1) at (79.421bp,17.892bp) [draw,draw=none] {$\mathtt{2314}$};
  \node (node_0) at (174.42bp,315.89bp) [draw,draw=none] {$\mathtt{1243}$};
  \draw [green,->] (node_0) ..controls (124.95bp,307.28bp) and (88.338bp,299.05bp)  .. (78.421bp,288.89bp) .. controls (71.149bp,281.44bp) and (68.102bp,270.23bp)  .. (node_3);
  \definecolor{strokecol}{rgb}{0.0,0.0,0.0};
  \pgfsetstrokecolor{strokecol}
  \draw (87.421bp,278.89bp) node {$3$};
  \draw [green,->] (node_2) ..controls (44.272bp,181.29bp) and (50.494bp,187.84bp)  .. (54.421bp,194.89bp) .. controls (59.205bp,203.48bp) and (62.121bp,214.01bp)  .. (node_3);
  \draw (71.421bp,204.89bp) node {$3$};
  \draw [red,->] (node_2) ..controls (26.228bp,186.32bp) and (25.919bp,202.77bp)  .. (32.421bp,214.89bp) .. controls (34.833bp,219.39bp) and (38.346bp,223.39bp)  .. (node_3);
  \draw (41.421bp,204.89bp) node {$2$};
  \draw [blue,->] (node_2) ..controls (62.285bp,174.98bp) and (70.421bp,172.74bp)  .. (70.421bp,167.89bp) .. controls (70.421bp,164.94bp) and (67.4bp,162.95bp)  .. (node_2);
  \draw (79.421bp,167.89bp) node {$1$};
  \draw [blue,->] (node_3) ..controls (28.997bp,234.17bp) and (13.572bp,227.41bp)  .. (5.4214bp,214.89bp) .. controls (-0.87737bp,205.22bp) and (4.7469bp,193.56bp)  .. (node_2);
  \draw (14.421bp,204.89bp) node {$1$};
  \draw [red,->] (node_3) ..controls (100.28bp,248.98bp) and (108.42bp,246.74bp)  .. (108.42bp,241.89bp) .. controls (108.42bp,238.94bp) and (105.4bp,236.95bp)  .. (node_3);
  \draw (117.42bp,241.89bp) node {$2$};
  \draw [red,->] (node_4) ..controls (215.28bp,100.37bp) and (223.42bp,98.013bp)  .. (223.42bp,92.892bp) .. controls (223.42bp,89.772bp) and (220.4bp,87.676bp)  .. (node_4);
  \draw (232.42bp,92.892bp) node {$2$};
  \draw [yellow,->] (node_3) ..controls (89.622bp,262.35bp) and (113.03bp,282.62bp)  .. (122.42bp,288.89bp) .. controls (129.96bp,293.93bp) and (138.55bp,298.73bp)  .. (node_0);
  \draw (131.42bp,278.89bp) node {$4$};
  \draw [yellow,->] (node_1) ..controls (48.18bp,30.992bp) and (37.016bp,37.219bp)  .. (28.421bp,44.892bp) .. controls (13.156bp,58.521bp) and (8.4693bp,63.343bp)  .. (2.4214bp,82.892bp) .. controls (-4.7462bp,106.06bp) and (6.6849bp,132.82bp)  .. (node_2);
  \draw (11.421bp,92.892bp) node {$4$};
  \draw [blue,->] (node_1) ..controls (68.389bp,50.34bp) and (45.835bp,116.67bp)  .. (node_2);
  \draw (66.421bp,92.892bp) node {$1$};
  \draw [yellow,->] (node_0) ..controls (197.91bp,302.81bp) and (205.3bp,296.61bp)  .. (209.42bp,288.89bp) .. controls (224.95bp,259.82bp) and (213.31bp,247.84bp)  .. (214.42bp,214.89bp) .. controls (215.83bp,173.08bp) and (226.73bp,158.98bp)  .. (209.42bp,120.89bp) .. controls (207.42bp,116.48bp) and (204.39bp,112.37bp)  .. (node_4);
  \draw (224.42bp,204.89bp) node {$4$};
  \draw [red,->] (node_0) ..controls (208.64bp,304.72bp) and (219.69bp,298.33bp)  .. (226.42bp,288.89bp) .. controls (245.73bp,261.83bp) and (234.97bp,248.05bp)  .. (237.42bp,214.89bp) .. controls (240.52bp,172.94bp) and (249.66bp,155.95bp)  .. (226.42bp,120.89bp) .. controls (222.75bp,115.36bp) and (217.57bp,110.72bp)  .. (node_4);
  \draw (247.42bp,204.89bp) node {$2$};
  \draw [green,->] (node_3) ..controls (76.728bp,228.11bp) and (81.228bp,221.36bp)  .. (84.421bp,214.89bp) .. controls (92.327bp,198.89bp) and (94.493bp,194.5bp)  .. (97.421bp,176.89bp) .. controls (105.83bp,126.35bp) and (92.778bp,66.104bp)  .. (node_1);
  \draw (108.42bp,130.89bp) node {$3$};
  \draw [green,->] (node_4) ..controls (164.21bp,74.14bp) and (147.23bp,56.919bp)  .. (130.42bp,44.892bp) .. controls (123.35bp,39.836bp) and (115.22bp,35.107bp)  .. (node_1);
  \draw (162.42bp,54.892bp) node {$3$};
  \draw [green,->] (node_1) ..controls (113.28bp,21.436bp) and (121.42bp,20.318bp)  .. (121.42bp,17.892bp) .. controls (121.42bp,16.414bp) and (118.4bp,15.422bp)  .. (node_1);
  \draw (130.42bp,17.892bp) node {$3$};
  \draw [red,->] (node_1) ..controls (109.88bp,40.928bp) and (139.42bp,38.088bp)  .. (139.42bp,17.892bp) .. controls (139.42bp,0.45816bp) and (117.41bp,-4.0419bp)  .. (node_1);
  \draw (148.42bp,17.892bp) node {$2$};
  \draw [yellow,->] (node_4) ..controls (172.62bp,118.86bp) and (159.86bp,159.53bp)  .. (155.42bp,194.89bp) .. controls (150.85bp,231.32bp) and (161.09bp,273.65bp)  .. (node_0);
  \draw (164.42bp,204.89bp) node {$4$};
  \draw [blue,->] (node_4) ..controls (181.59bp,111.37bp) and (181.65bp,127.23bp)  .. (181.42bp,140.89bp) .. controls (180.31bp,206.7bp) and (179.96bp,223.17bp)  .. (176.42bp,288.89bp) .. controls (176.28bp,291.6bp) and (176.09bp,294.46bp)  .. (node_0);
  \draw (188.42bp,204.89bp) node {$1$};
  \draw [blue,->] (node_0) ..controls (208.28bp,322.98bp) and (216.42bp,320.74bp)  .. (216.42bp,315.89bp) .. controls (216.42bp,312.94bp) and (213.4bp,310.95bp)  .. (node_0);
  \draw (225.42bp,315.89bp) node {$1$};
  \draw [yellow,->] (node_2) ..controls (47.275bp,147.21bp) and (67.3bp,126.09bp)  .. (76.421bp,120.89bp) .. controls (98.258bp,108.46bp) and (125.8bp,101.38bp)  .. (node_4);
  \draw (85.421bp,130.89bp) node {$4$};
\end{tikzpicture}
\caption{The promotion Markov chain digraph for the poset in Example~\ref{example.promotion}.
\label{figure.promotion}}
\end{figure}

We may represent the promotion operator $\partial_i$ by a $|\mathcal{L}(P)| \times |\mathcal{L}(P)|$-dimensional matrix,
where row $k$ and column $j$ contains 1 if the $j$-th linear extension in~\eqref{equation.LP example} is mapped
to the $k$-th linear extension in~\eqref{equation.LP example} under $\partial_i$; the rest of the entries are zero. 
For example, $\partial_1$ is represented by the matrix
\[
	\begin{pmatrix}
	1&0&1&0&1\\
	0&1&0&1&0\\
	0&0&0&0&0\\
	0&0&0&0&0\\
	0&0&0&0&0
	\end{pmatrix}.
\]
The right Cayley graph of the monoid generated by the matrices for the promotion operators
$\partial_1,\partial_2,\partial_3,\partial_4$ is depicted in Figure~\ref{figure.right Cayley example}.
The vertices in the right Cayley graph are labeled by reduced words in the generators.
For example $[1,4,1]$ stands for the element $\partial_1 \partial_4 \partial_1$.

\begin{figure}
\rotatebox{270}{\scalebox{0.5}{
\begin{tikzpicture}[>=latex,line join=bevel,]
\node (node_13) at (419.0bp,170.6bp) [draw,draw=none] {$[4, 4]$};
  \node (node_14) at (59.0bp,19.599bp) [draw,draw=none] {$[1, 3]$};
  \node (node_18) at (1105.0bp,19.599bp) [draw,draw=none] {$[3, 2]$};
  \node (node_19) at (1027.0bp,94.599bp) [draw,draw=none] {$[3, 4]$};
  \node (node_9) at (419.0bp,322.6bp) [draw,draw=none] {$[4]$};
  \node (node_8) at (709.0bp,94.599bp) [draw,draw=none] {$[2, 4, 3, 4]$};
  \node (node_7) at (611.0bp,19.599bp) [draw,draw=none] {$[4, 1]$};
  \node (node_6) at (175.0bp,170.6bp) [draw,draw=none] {$[4, 3]$};
  \node (node_5) at (154.0bp,94.599bp) [draw,draw=none] {$[4, 3, 4]$};
  \node (node_4) at (751.0bp,170.6bp) [draw,draw=none] {$[2, 4, 3]$};
  \node (node_3) at (253.0bp,246.6bp) [draw,draw=none] {$[1, 4]$};
  \node (node_2) at (726.0bp,322.6bp) [draw,draw=none] {$[2, 1]$};
  \node (node_1) at (353.0bp,94.599bp) [draw,draw=none] {$[4, 4, 4]$};
  \node (node_0) at (585.0bp,473.6bp) [draw,draw=none] {$[]$};
  \node (node_17) at (677.0bp,246.6bp) [draw,draw=none] {$[2, 4]$};
  \node (node_11) at (154.0bp,322.6bp) [draw,draw=none] {$[1]$};
  \node (node_16) at (698.0bp,398.6bp) [draw,draw=none] {$[2]$};
  \node (node_10) at (917.0bp,19.599bp) [draw,draw=none] {$[3, 1]$};
  \node (node_15) at (266.0bp,19.599bp) [draw,draw=none] {$[1, 4, 1]$};
  \node (node_12) at (1005.0bp,170.6bp) [draw,draw=none] {$[3]$};
  \draw [black,->] (node_16) ..controls (705.6bp,377.97bp) and (713.58bp,356.32bp)  .. (node_2);
  \definecolor{strokecol}{rgb}{0.0,0.0,0.0};
  \pgfsetstrokecolor{strokecol}
  \draw (724.0bp,360.6bp) node {$1$};
  \draw [black,->] (node_9) ..controls (375.94bp,313.87bp) and (243.2bp,285.34bp)  .. (213.0bp,256.6bp) .. controls (193.72bp,238.26bp) and (183.59bp,208.4bp)  .. (node_6);
  \draw (222.0bp,246.6bp) node {$3$};
  \draw [black,->] (node_2) ..controls (748.1bp,326.79bp) and (757.0bp,325.72bp)  .. (757.0bp,322.6bp) .. controls (757.0bp,320.7bp) and (753.69bp,319.56bp)  .. (node_2);
  \draw (766.0bp,322.6bp) node {$2$};
  \draw [black,->] (node_2) ..controls (750.87bp,345.88bp) and (775.0bp,343.01bp)  .. (775.0bp,322.6bp) .. controls (775.0bp,305.38bp) and (757.82bp,300.64bp)  .. (node_2);
  \draw (784.0bp,322.6bp) node {$1$};
  \draw [black,->] (node_5) ..controls (148.66bp,75.604bp) and (146.02bp,57.778bp)  .. (155.0bp,46.599bp) .. controls (156.33bp,44.942bp) and (205.66bp,33.395bp)  .. (node_15);
  \draw (164.0bp,56.599bp) node {$3$};
  \draw [black,->] (node_5) ..controls (166.8bp,75.428bp) and (180.39bp,57.511bp)  .. (196.0bp,46.599bp) .. controls (208.68bp,37.739bp) and (224.65bp,31.289bp)  .. (node_15);
  \draw (205.0bp,56.599bp) node {$2$};
  \draw [black,->] (node_16) ..controls (693.69bp,377.55bp) and (688.97bp,353.35bp)  .. (686.0bp,332.6bp) .. controls (682.69bp,309.44bp) and (680.07bp,282.6bp)  .. (node_17);
  \draw (695.0bp,322.6bp) node {$4$};
  \draw [black,->] (node_3) ..controls (275.1bp,254.98bp) and (284.0bp,252.83bp)  .. (284.0bp,246.6bp) .. controls (284.0bp,242.8bp) and (280.69bp,240.52bp)  .. (node_3);
  \draw (293.0bp,246.6bp) node {$2$};
  \draw [black,->] (node_12) ..controls (1001.1bp,151.94bp) and (999.34bp,135.63bp)  .. (1004.0bp,122.6bp) .. controls (1005.4bp,118.58bp) and (1007.7bp,114.73bp)  .. (node_19);
  \draw (1013.0bp,132.6bp) node {$4$};
  \draw [black,->] (node_11) ..controls (177.61bp,304.48bp) and (211.86bp,278.18bp)  .. (node_3);
  \draw (224.0bp,284.6bp) node {$4$};
  \draw [black,->] (node_17) ..controls (655.34bp,235.03bp) and (644.76bp,227.7bp)  .. (638.0bp,218.6bp) .. controls (616.71bp,189.94bp) and (618.0bp,177.79bp)  .. (612.0bp,142.6bp) .. controls (605.77bp,106.07bp) and (607.53bp,62.552bp)  .. (node_7);
  \draw (621.0bp,132.6bp) node {$4$};
  \draw [black,->] (node_17) ..controls (668.37bp,209.45bp) and (646.68bp,119.22bp)  .. (622.0bp,46.599bp) .. controls (621.02bp,43.724bp) and (619.9bp,40.713bp)  .. (node_7);
  \draw (659.0bp,132.6bp) node {$1$};
  \draw [black,->] (node_8) ..controls (700.13bp,113.51bp) and (694.9bp,130.36bp)  .. (702.0bp,142.6bp) .. controls (706.79bp,150.85bp) and (715.02bp,156.87bp)  .. (node_4);
  \draw (711.0bp,132.6bp) node {$4$};
  \draw [black,->] (node_4) ..controls (763.82bp,156.27bp) and (770.28bp,149.03bp)  .. (776.0bp,142.6bp) .. controls (813.87bp,100.01bp) and (815.55bp,80.982bp)  .. (861.0bp,46.599bp) .. controls (871.05bp,39.0bp) and (883.57bp,32.718bp)  .. (node_10);
  \draw (829.0bp,94.599bp) node {$3$};
  \draw [black,->] (node_4) ..controls (776.39bp,156.71bp) and (788.14bp,149.71bp)  .. (798.0bp,142.6bp) .. controls (841.98bp,110.88bp) and (883.72bp,61.884bp)  .. (node_10);
  \draw (870.0bp,94.599bp) node {$2$};
  \draw [black,->] (node_3) ..controls (236.33bp,232.57bp) and (228.04bp,225.33bp)  .. (221.0bp,218.6bp) .. controls (210.29bp,208.36bp) and (198.75bp,196.32bp)  .. (node_6);
  \draw (230.0bp,208.6bp) node {$3$};
  \draw [black,->] (node_8) ..controls (735.55bp,75.161bp) and (762.92bp,56.841bp)  .. (789.0bp,46.599bp) .. controls (824.15bp,32.8bp) and (867.63bp,25.536bp)  .. (node_10);
  \draw (798.0bp,56.599bp) node {$1$};
  \draw [black,->] (node_11) ..controls (113.73bp,311.95bp) and (0.0bp,279.4bp)  .. (0.0bp,246.6bp) .. controls (0.0bp,246.6bp) and (0.0bp,246.6bp)  .. (0.0bp,94.599bp) .. controls (0.0bp,69.488bp) and (20.685bp,47.819bp)  .. (node_14);
  \draw (9.0bp,170.6bp) node {$3$};
  \draw [black,->] (node_1) ..controls (331.63bp,80.792bp) and (321.44bp,73.682bp)  .. (313.0bp,66.599bp) .. controls (301.65bp,57.077bp) and (289.86bp,45.284bp)  .. (node_15);
  \draw (322.0bp,56.599bp) node {$3$};
  \draw [black,->] (node_1) ..controls (348.38bp,75.451bp) and (342.35bp,57.546bp)  .. (331.0bp,46.599bp) .. controls (320.8bp,36.759bp) and (306.5bp,30.293bp)  .. (node_15);
  \draw (352.0bp,56.599bp) node {$2$};
  \draw [black,->] (node_5) ..controls (132.86bp,80.498bp) and (122.68bp,73.384bp)  .. (114.0bp,66.599bp) .. controls (103.33bp,58.259bp) and (101.29bp,55.412bp)  .. (91.0bp,46.599bp) .. controls (86.515bp,42.757bp) and (81.653bp,38.634bp)  .. (node_14);
  \draw (123.0bp,56.599bp) node {$1$};
  \draw [black,->] (node_14) ..controls (81.1bp,21.584bp) and (90.0bp,21.076bp)  .. (90.0bp,19.599bp) .. controls (90.0bp,18.7bp) and (86.695bp,18.16bp)  .. (node_14);
  \draw (99.0bp,19.599bp) node {$4$};
  \draw [black,->] (node_14) ..controls (85.173bp,40.277bp) and (108.0bp,37.363bp)  .. (108.0bp,19.599bp) .. controls (108.0bp,4.8888bp) and (92.345bp,0.36183bp)  .. (node_14);
  \draw (117.0bp,19.599bp) node {$3$};
  \draw [black,->] (node_14) ..controls (93.013bp,42.635bp) and (126.0bp,39.795bp)  .. (126.0bp,19.599bp) .. controls (126.0bp,1.9285bp) and (100.74bp,-2.4547bp)  .. (node_14);
  \draw (135.0bp,19.599bp) node {$2$};
  \draw [black,->] (node_14) ..controls (101.0bp,44.922bp) and (144.0bp,42.03bp)  .. (144.0bp,19.599bp) .. controls (144.0bp,-0.64079bp) and (108.99bp,-4.9717bp)  .. (node_14);
  \draw (153.0bp,19.599bp) node {$1$};
  \draw [black,->] (node_12) ..controls (1022.6bp,162.3bp) and (1038.0bp,153.71bp)  .. (1048.0bp,142.6bp) .. controls (1076.1bp,111.4bp) and (1092.7bp,63.91bp)  .. (node_18);
  \draw (1092.0bp,94.599bp) node {$3$};
  \draw [black,->] (node_12) ..controls (1034.0bp,162.45bp) and (1087.7bp,143.56bp)  .. (1105.0bp,104.6bp) .. controls (1114.4bp,83.374bp) and (1112.3bp,56.11bp)  .. (node_18);
  \draw (1119.0bp,94.599bp) node {$2$};
  \draw [black,->] (node_2) ..controls (712.55bp,301.74bp) and (698.2bp,279.48bp)  .. (node_17);
  \draw (716.0bp,284.6bp) node {$4$};
  \draw [black,->] (node_9) ..controls (419.0bp,290.06bp) and (419.0bp,222.32bp)  .. (node_13);
  \draw (428.0bp,246.6bp) node {$4$};
  \draw [black,->] (node_0) ..controls (533.46bp,455.54bp) and (247.05bp,355.2bp)  .. (node_11);
  \draw (407.0bp,398.6bp) node {$1$};
  \draw [black,->] (node_13) ..controls (430.94bp,156.39bp) and (436.89bp,149.15bp)  .. (442.0bp,142.6bp) .. controls (474.63bp,100.75bp) and (469.24bp,76.626bp)  .. (513.0bp,46.599bp) .. controls (535.55bp,31.124bp) and (566.72bp,24.49bp)  .. (node_7);
  \draw (485.0bp,94.599bp) node {$3$};
  \draw [black,->] (node_13) ..controls (441.02bp,157.63bp) and (452.57bp,150.21bp)  .. (462.0bp,142.6bp) .. controls (509.21bp,104.52bp) and (509.97bp,82.313bp)  .. (559.0bp,46.599bp) .. controls (568.15bp,39.932bp) and (579.18bp,33.957bp)  .. (node_7);
  \draw (526.0bp,94.599bp) node {$2$};
  \draw [black,->] (node_5) ..controls (138.16bp,113.14bp) and (128.13bp,129.4bp)  .. (135.0bp,142.6bp) .. controls (138.96bp,150.21bp) and (146.05bp,156.18bp)  .. (node_6);
  \draw (144.0bp,132.6bp) node {$4$};
  \draw [black,->] (node_19) ..controls (1027.2bp,113.15bp) and (1026.4bp,129.42bp)  .. (1022.0bp,142.6bp) .. controls (1020.8bp,146.16bp) and (1019.1bp,149.74bp)  .. (node_12);
  \draw (1035.0bp,132.6bp) node {$4$};
  \draw [black,->] (node_16) ..controls (739.35bp,390.05bp) and (861.68bp,360.9bp)  .. (935.0bp,294.6bp) .. controls (943.81bp,286.63bp) and (978.85bp,220.73bp)  .. (node_12);
  \draw (957.0bp,284.6bp) node {$3$};
  \draw [black,->] (node_12) ..controls (989.61bp,141.75bp) and (961.14bp,89.147bp)  .. (935.0bp,46.599bp) .. controls (933.05bp,43.431bp) and (930.91bp,40.102bp)  .. (node_10);
  \draw (977.0bp,94.599bp) node {$1$};
  \draw [black,->] (node_18) ..controls (1127.1bp,21.584bp) and (1136.0bp,21.076bp)  .. (1136.0bp,19.599bp) .. controls (1136.0bp,18.7bp) and (1132.7bp,18.16bp)  .. (node_18);
  \draw (1145.0bp,19.599bp) node {$4$};
  \draw [black,->] (node_18) ..controls (1131.2bp,40.277bp) and (1154.0bp,37.363bp)  .. (1154.0bp,19.599bp) .. controls (1154.0bp,4.8888bp) and (1138.3bp,0.36183bp)  .. (node_18);
  \draw (1163.0bp,19.599bp) node {$3$};
  \draw [black,->] (node_18) ..controls (1139.0bp,42.635bp) and (1172.0bp,39.795bp)  .. (1172.0bp,19.599bp) .. controls (1172.0bp,1.9285bp) and (1146.7bp,-2.4547bp)  .. (node_18);
  \draw (1181.0bp,19.599bp) node {$2$};
  \draw [black,->] (node_18) ..controls (1147.0bp,44.922bp) and (1190.0bp,42.03bp)  .. (1190.0bp,19.599bp) .. controls (1190.0bp,-0.64079bp) and (1155.0bp,-4.9717bp)  .. (node_18);
  \draw (1199.0bp,19.599bp) node {$1$};
  \draw [black,->] (node_17) ..controls (699.1bp,254.98bp) and (708.0bp,252.83bp)  .. (708.0bp,246.6bp) .. controls (708.0bp,242.8bp) and (704.69bp,240.52bp)  .. (node_17);
  \draw (717.0bp,246.6bp) node {$2$};
  \draw [black,->] (node_6) ..controls (169.33bp,150.08bp) and (163.43bp,128.73bp)  .. (node_5);
  \draw (175.0bp,132.6bp) node {$4$};
  \draw [black,->] (node_1) ..controls (343.64bp,113.51bp) and (337.89bp,130.8bp)  .. (346.0bp,142.6bp) .. controls (351.9bp,151.19bp) and (376.99bp,159.55bp)  .. (node_13);
  \draw (355.0bp,132.6bp) node {$4$};
  \draw [black,->] (node_6) ..controls (143.14bp,166.6bp) and (109.16bp,159.77bp)  .. (86.0bp,142.6bp) .. controls (69.559bp,130.42bp) and (66.163bp,124.11bp)  .. (60.0bp,104.6bp) .. controls (53.073bp,82.666bp) and (54.21bp,55.933bp)  .. (node_14);
  \draw (69.0bp,94.599bp) node {$3$};
  \draw [black,->] (node_6) ..controls (151.76bp,159.91bp) and (138.38bp,152.26bp)  .. (129.0bp,142.6bp) .. controls (98.447bp,111.13bp) and (76.469bp,63.749bp)  .. (node_14);
  \draw (109.0bp,94.599bp) node {$2$};
  \draw [black,->] (node_9) ..controls (386.62bp,307.77bp) and (312.71bp,273.93bp)  .. (node_3);
  \draw (364.0bp,284.6bp) node {$2$};
  \draw [black,->] (node_0) ..controls (607.9bp,458.4bp) and (655.57bp,426.76bp)  .. (node_16);
  \draw (663.0bp,436.6bp) node {$2$};
  \draw [black,->] (node_3) ..controls (252.23bp,217.6bp) and (251.32bp,165.79bp)  .. (254.0bp,122.6bp) .. controls (255.85bp,92.884bp) and (260.31bp,58.578bp)  .. (node_15);
  \draw (263.0bp,132.6bp) node {$4$};
  \draw [black,->] (node_3) ..controls (260.8bp,232.67bp) and (264.2bp,225.44bp)  .. (266.0bp,218.6bp) .. controls (281.13bp,160.97bp) and (273.4bp,144.08bp)  .. (270.0bp,84.599bp) .. controls (269.11bp,69.059bp) and (268.02bp,51.435bp)  .. (node_15);
  \draw (283.0bp,132.6bp) node {$1$};
  \draw [black,->] (node_13) ..controls (401.33bp,156.61bp) and (393.33bp,149.61bp)  .. (387.0bp,142.6bp) .. controls (378.38bp,133.05bp) and (370.01bp,121.27bp)  .. (node_1);
  \draw (396.0bp,132.6bp) node {$4$};
  \draw [black,->] (node_4) ..controls (742.74bp,156.3bp) and (738.59bp,149.06bp)  .. (735.0bp,142.6bp) .. controls (729.38bp,132.5bp) and (723.22bp,121.14bp)  .. (node_8);
  \draw (744.0bp,132.6bp) node {$4$};
  \draw [black,->] (node_0) ..controls (643.51bp,469.51bp) and (1005.0bp,442.58bp)  .. (1005.0bp,398.6bp) .. controls (1005.0bp,398.6bp) and (1005.0bp,398.6bp)  .. (1005.0bp,246.6bp) .. controls (1005.0bp,227.25bp) and (1005.0bp,205.14bp)  .. (node_12);
  \draw (1014.0bp,322.6bp) node {$3$};
  \draw [black,->] (node_19) ..controls (1041.5bp,75.968bp) and (1055.4bp,59.249bp)  .. (1069.0bp,46.599bp) .. controls (1073.6bp,42.311bp) and (1078.9bp,38.043bp)  .. (node_18);
  \draw (1078.0bp,56.599bp) node {$1$};
  \draw [black,->] (node_9) ..controls (450.93bp,272.21bp) and (564.09bp,93.631bp)  .. (node_7);
  \draw (529.0bp,170.6bp) node {$1$};
  \draw [black,->] (node_15) ..controls (293.54bp,21.478bp) and (302.0bp,20.935bp)  .. (302.0bp,19.599bp) .. controls (302.0bp,18.785bp) and (298.86bp,18.266bp)  .. (node_15);
  \draw (311.0bp,19.599bp) node {$4$};
  \draw [black,->] (node_15) ..controls (294.61bp,40.304bp) and (320.0bp,37.446bp)  .. (320.0bp,19.599bp) .. controls (320.0bp,4.5412bp) and (301.92bp,0.15358bp)  .. (node_15);
  \draw (329.0bp,19.599bp) node {$3$};
  \draw [black,->] (node_15) ..controls (302.55bp,42.635bp) and (338.0bp,39.795bp)  .. (338.0bp,19.599bp) .. controls (338.0bp,1.7707bp) and (310.37bp,-2.5319bp)  .. (node_15);
  \draw (347.0bp,19.599bp) node {$2$};
  \draw [black,->] (node_15) ..controls (310.47bp,44.922bp) and (356.0bp,42.03bp)  .. (356.0bp,19.599bp) .. controls (356.0bp,-0.72841bp) and (318.61bp,-5.0089bp)  .. (node_15);
  \draw (365.0bp,19.599bp) node {$1$};
  \draw [black,->] (node_4) ..controls (762.3bp,137.01bp) and (784.82bp,64.154bp)  .. (770.0bp,46.599bp) .. controls (761.34bp,36.339bp) and (676.35bp,26.293bp)  .. (node_7);
  \draw (782.0bp,94.599bp) node {$1$};
  \draw [black,->] (node_19) ..controls (1022.7bp,75.085bp) and (1016.9bp,56.988bp)  .. (1005.0bp,46.599bp) .. controls (986.93bp,30.77bp) and (959.74bp,24.221bp)  .. (node_10);
  \draw (1027.0bp,56.599bp) node {$3$};
  \draw [black,->] (node_19) ..controls (1004.7bp,82.058bp) and (992.65bp,74.542bp)  .. (983.0bp,66.599bp) .. controls (973.54bp,58.806bp) and (973.64bp,54.177bp)  .. (964.0bp,46.599bp) .. controls (956.45bp,40.664bp) and (947.45bp,35.169bp)  .. (node_10);
  \draw (992.0bp,56.599bp) node {$2$};
  \draw [black,->] (node_17) ..controls (697.64bp,225.4bp) and (720.22bp,202.21bp)  .. (node_4);
  \draw (731.0bp,208.6bp) node {$3$};
  \draw [black,->] (node_2) ..controls (782.82bp,312.84bp) and (921.0bp,285.75bp)  .. (921.0bp,246.6bp) .. controls (921.0bp,246.6bp) and (921.0bp,246.6bp)  .. (921.0bp,94.599bp) .. controls (921.0bp,75.48bp) and (919.69bp,53.659bp)  .. (node_10);
  \draw (930.0bp,170.6bp) node {$3$};
  \draw [black,->] (node_16) ..controls (714.65bp,407.43bp) and (724.0bp,405.58bp)  .. (724.0bp,398.6bp) .. controls (724.0bp,394.46bp) and (720.7bp,392.12bp)  .. (node_16);
  \draw (733.0bp,398.6bp) node {$2$};
  \draw [black,->] (node_10) ..controls (939.1bp,21.584bp) and (948.0bp,21.076bp)  .. (948.0bp,19.599bp) .. controls (948.0bp,18.7bp) and (944.69bp,18.16bp)  .. (node_10);
  \draw (957.0bp,19.599bp) node {$4$};
  \draw [black,->] (node_10) ..controls (943.17bp,40.277bp) and (966.0bp,37.363bp)  .. (966.0bp,19.599bp) .. controls (966.0bp,4.8888bp) and (950.35bp,0.36183bp)  .. (node_10);
  \draw (975.0bp,19.599bp) node {$3$};
  \draw [black,->] (node_10) ..controls (951.01bp,42.635bp) and (984.0bp,39.795bp)  .. (984.0bp,19.599bp) .. controls (984.0bp,1.9285bp) and (958.74bp,-2.4547bp)  .. (node_10);
  \draw (993.0bp,19.599bp) node {$2$};
  \draw [black,->] (node_10) ..controls (959.0bp,44.922bp) and (1002.0bp,42.03bp)  .. (1002.0bp,19.599bp) .. controls (1002.0bp,-0.64079bp) and (966.99bp,-4.9717bp)  .. (node_10);
  \draw (1011.0bp,19.599bp) node {$1$};
  \draw [black,->] (node_6) ..controls (191.62bp,137.6bp) and (227.26bp,67.644bp)  .. (242.0bp,46.599bp) .. controls (244.59bp,42.901bp) and (247.62bp,39.147bp)  .. (node_15);
  \draw (228.0bp,94.599bp) node {$1$};
  \draw [black,->] (node_0) ..controls (554.9bp,446.22bp) and (470.11bp,369.09bp)  .. (node_9);
  \draw (522.0bp,398.6bp) node {$4$};
  \draw [black,->] (node_7) ..controls (633.1bp,21.584bp) and (642.0bp,21.076bp)  .. (642.0bp,19.599bp) .. controls (642.0bp,18.7bp) and (638.69bp,18.16bp)  .. (node_7);
  \draw (651.0bp,19.599bp) node {$4$};
  \draw [black,->] (node_7) ..controls (637.17bp,40.277bp) and (660.0bp,37.363bp)  .. (660.0bp,19.599bp) .. controls (660.0bp,4.8888bp) and (644.35bp,0.36183bp)  .. (node_7);
  \draw (669.0bp,19.599bp) node {$3$};
  \draw [black,->] (node_7) ..controls (645.01bp,42.635bp) and (678.0bp,39.795bp)  .. (678.0bp,19.599bp) .. controls (678.0bp,1.9285bp) and (652.74bp,-2.4547bp)  .. (node_7);
  \draw (687.0bp,19.599bp) node {$2$};
  \draw [black,->] (node_7) ..controls (653.0bp,44.922bp) and (696.0bp,42.03bp)  .. (696.0bp,19.599bp) .. controls (696.0bp,-0.64079bp) and (660.99bp,-4.9717bp)  .. (node_7);
  \draw (705.0bp,19.599bp) node {$1$};
  \draw [black,->] (node_11) ..controls (170.65bp,327.01bp) and (180.0bp,326.09bp)  .. (180.0bp,322.6bp) .. controls (180.0bp,320.53bp) and (176.7bp,319.36bp)  .. (node_11);
  \draw (189.0bp,322.6bp) node {$2$};
  \draw [black,->] (node_11) ..controls (176.34bp,345.88bp) and (198.0bp,343.01bp)  .. (198.0bp,322.6bp) .. controls (198.0bp,305.7bp) and (183.14bp,300.82bp)  .. (node_11);
  \draw (207.0bp,322.6bp) node {$1$};
  \draw [black,->] (node_13) ..controls (428.3bp,139.95bp) and (442.23bp,80.016bp)  .. (413.0bp,46.599bp) .. controls (397.77bp,29.195bp) and (332.49bp,22.937bp)  .. (node_15);
  \draw (439.0bp,94.599bp) node {$1$};
  \draw [black,->] (node_8) ..controls (705.01bp,75.249bp) and (699.47bp,57.239bp)  .. (688.0bp,46.599bp) .. controls (673.18bp,32.849bp) and (650.94bp,26.038bp)  .. (node_7);
  \draw (710.0bp,56.599bp) node {$3$};
  \draw [black,->] (node_8) ..controls (685.51bp,80.894bp) and (674.76bp,73.894bp)  .. (666.0bp,66.599bp) .. controls (656.58bp,58.756bp) and (656.16bp,54.751bp)  .. (647.0bp,46.599bp) .. controls (642.21bp,42.34bp) and (636.8bp,38.04bp)  .. (node_7);
  \draw (675.0bp,56.599bp) node {$2$};
  \draw [black,->] (node_1) ..controls (381.98bp,75.031bp) and (412.34bp,56.367bp)  .. (441.0bp,46.599bp) .. controls (491.57bp,29.366bp) and (554.42bp,23.034bp)  .. (node_7);
  \draw (450.0bp,56.599bp) node {$1$};
\end{tikzpicture}
}}
\caption{Right Cayley graph for the promotion Markov chain of Example~\ref{example.promotion}.
\label{figure.right Cayley example}}
\end{figure}
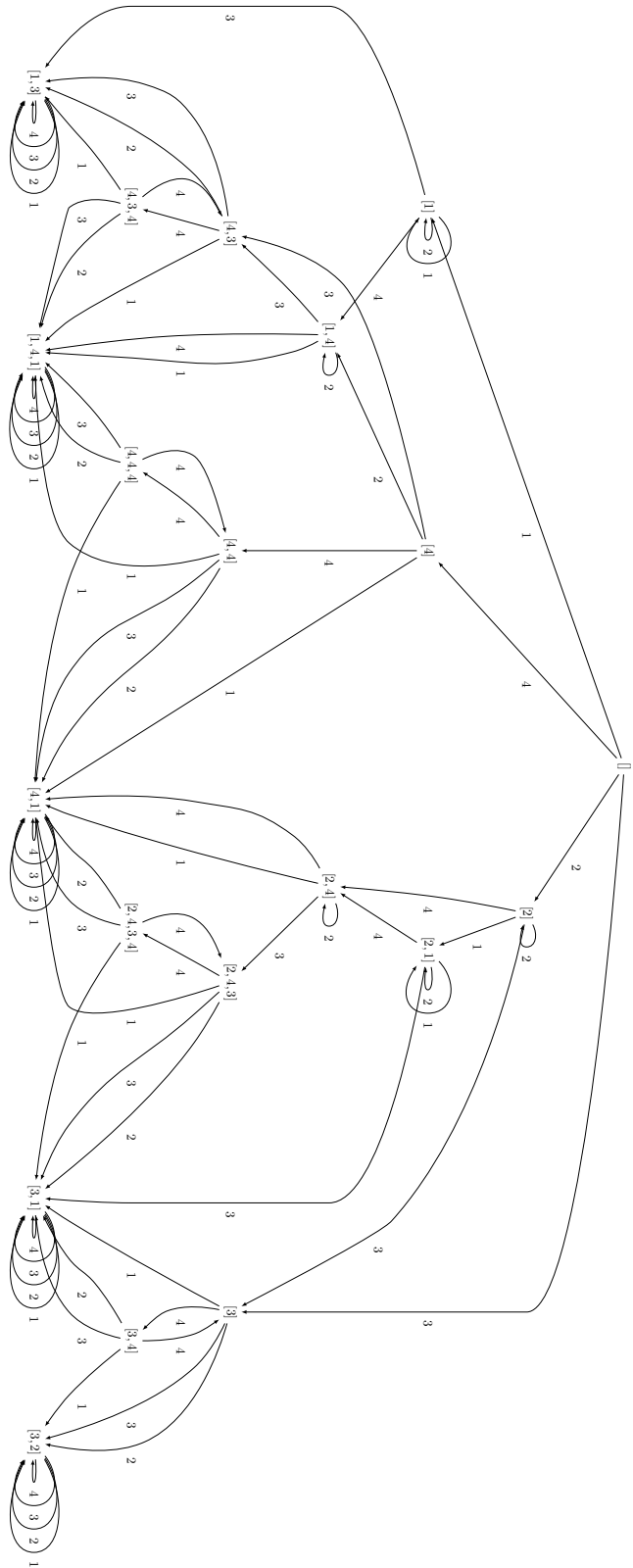
\end{example}

We prove some useful properties of the right Cayley graph of the semigroup $S$ generated by $\partial_i$ for 
$1\leqslant i \leqslant n$.

\begin{proposition}
\label{proposition.length reduced}
Any element in $K(S)$ can be written as $\partial_{w_1} \cdots \partial_{w_{n-1}}$, where 
$w_1,\ldots,w_{n-1} \in \{1,2,\ldots,n\}$ are distinct. In particular, the length of any reduced word for the elements 
in $K(S)$ is less than $n$.
\end{proposition}

\begin{proof}
Each element in $K(S)$ corresponds to a linear extension in $\mathcal{L}(P)$. For a given $\pi \in \mathcal{L}(P)$,
we now construct a word $w_1 \ldots w_{n-1}$ with distinct letters such that $\pi=\partial_{w_1} \cdots \partial_{w_{n-1}} \pi'$
for all $\pi' \in \mathcal{L}(P)$. In particular, this means that $\partial_{w_1} \cdots \partial_{w_{n-1}} \in K(S)$.

Write $\pi = \pi_1 \ldots \pi_n$ in one-line notation and set $\pi^{(1)} = \pi$.
Construct $\pi^{(m+1)}$ from $\pi^{(m)}$ for $1\leqslant m < n$ as follows.
Set $i^{(m)}_1=1$ and then recursively find the smallest $i^{(m)}_{j+1}>i^{(m)}_j$ 
such that $\pi^{(m)}_{i^{(m)}_j} \prec \pi^{(m)}_{i^{(m)}_{j+1}}$ if possible. If there is no such 
$i^{(m)}_{j+1}$, set $k^{(m)}=j$. Define $w_m=\pi^{(m)}_{i^{(m)}_{k^{(m)}}}$. Next construct $\pi^{(m+1)}$ from $\pi^{(m)}$ 
by removing $\pi^{(m)}_1$ and replacing $\pi^{(m)}_{i^{(m)}_j}$ by $\pi^{(m)}_{i^{(m)}_{j-1}}$ for 
$2\leqslant j \leqslant k^{(m)}$.

Next we show that $\pi = \partial_{w_1} \cdots \partial_{w_{n-1}} \pi'$ for any $\pi'\in \mathcal{L}(P)$, proving that
$\partial_{w_1} \cdots \partial_{w_{n-1}} \in K(S)$ corresponding to the linear extension $\pi$. We will do so by
induction on $n$. For $n=2$, $P$ is either the antichain with vertex 1 incomparable to vertex 2 or 2 is bigger than 1.
In the first case, there are two linear extension $\pi=12$ or $21$. The algorithm determines $w=\pi_1$ and indeed
$\partial_{\pi_1}(12)=\partial_{\pi_1}(21)=\pi_1\pi_2=\pi$. In the second case, there is only one linear extension
$\pi=12$ and the algorithm determines $w=2$. Indeed $\partial_2(12)=12$.

Now assume by induction that the algorithm works for posets with strictly less than $n$ vertices.
In particular, for $\pi^{(2)}$ from the algorithm $\pi^{(2)} = \partial_{w_2}\cdots \partial_{w_{n-1}} \pi'$
for any linear extension $\pi'$ of the poset $P'$ obtained from $P$ by deleting the vertex $w_1$. 
Also, by induction $w_2,\ldots,w_{n-1}$ are distinct and different from $w_1$. Note that $w_1$
is a maximal element in $P$. Hence for any linear extension $\pi'$ of $P$, we have that 
$\partial_{w_2}\cdots \partial_{w_{n-1}} \pi'$ is a linear extension of $P$ such that removing the letter $w_1$ 
results in $\pi^{(2)}$. Let $\sigma \in \mathcal{L}(P)$ be such a linear extension, that is, $\sigma \setminus w_1 = \pi^{(2)}$.
Consider the saturated chain $\pi_1=a_1 \prec a_2 \prec \cdots \prec a_k = w_1$ in $P$ from $\pi_1$ to $w_1$.
Such a chain exists by the definition of $w_1$.
In $\pi^{(2)}$ and hence also in $\sigma$ the letter $a_{k-1}$ is the rightmost letter that is covered in $P$ by $w_1$.
This is since by the algorithm to construct $\pi^{(2)}$, the letter $a_{k-1}$ replaced the letter $a_k=w_1$ in $\pi$.
In $\sigma$, the letter $w_1$ must sit to the right of the letter $a_{k-1}$ since $a_{k-1} \prec w_1$.
Hence, when acting with $\partial_{w_1}$ on $\sigma$, the letter $w_1$ interchanges with all letters to its left
until it reaches the letter $a_{k-1}$. By the action of $\tau_i$ as in~\eqref{equation.tau}, the letter $w_1$ will stay
in the position where $a_{k-1}$ was in $\sigma$ and then the letter $a_{k-1}$ starts moving left. The letter $a_{k-2}$
is the rightmost letter in $\sigma$ that is covered by $a_{k-1}$ in $P$, again by the definition of the algorithm.
The letter $a_{k-1}$ replaces the letter $a_{k-2}$ and $a_{k-2}$ starts moving left and so on. Finally, the letter
$a_1=\pi_1$ moves into first position. Hence $\partial_{w_1} \sigma = \pi$. This proves the claim.
\end{proof}

\begin{example}
Take the poset from Example~\ref{example.promotion} and the linear extension $\pi = 1243$. Set $\pi^{(1)}=\pi$.
The first sequence of increasing entries in $\pi^{(1)}$ is given by the underlined entries
\[
	\textcolor{darkred}{\underline{1}}2 \textcolor{darkred}{\underline{4}}3.
\]
Hence $w_1=4$ and $\pi^{(2)}=213$. The next sequence of increasing entries is given by
\[
	\textcolor{darkred}{\underline{2}} 1 \textcolor{darkred}{\underline{3}}.
\]
Hence $w_2=3$ and $\pi^{(3)}=12$. Next we find the increasing sequence $\textcolor{darkred}{\underline{1}}2$, so
that $w_3=1$. Indeed, comparing with Figure~\ref{figure.right Cayley example}, we see that 
\[
	\partial_4 \partial_3 \partial_1 = \partial_1 \partial_4 \partial_1
\]
is in $K(S)$.

Note that the above algorithm does not always give a shortest path to the ideal in the right Cayley graph.
For example, if $\pi=2143$ the algorithm gives
\[
	\textcolor{darkred}{\underline{2}}1\textcolor{darkred}{\underline{4}} 3 \to
	\textcolor{darkred}{\underline{1}}23 \to
	\textcolor{darkred}{\underline{2}}\textcolor{darkred}{\underline{3}} \to
	2,
\]
so that $w_1 w_2 w_3 = 413$. From Figure~\ref{figure.right Cayley example}, we see that
$\partial_4 \partial_1 \partial_3 = \partial_4 \partial_1$ is in $K(S)$.
\end{example}

The (unnormalized) stationary distribution of the promotion Markov chain was computed
in~\cite[Theorem 4.5]{AyyerKleeSchilling.2014}. Recall that our conventions are different 
from~\cite{AyyerKleeSchilling.2014}.

\begin{theorem}{\cite[Theorem 4.5]{AyyerKleeSchilling.2014}}
\label{theorem.stationary promotion}
The (unnormalized) stationary distribution for the promotion Markov chain $\Psi_\pi$ for
$\pi \in \mathcal{L}(P)$ for a finite poset $P$ with $n=|P|$ is given by
\begin{equation}
\label{equation.stationary linear extension}
	\Psi_\pi = \prod_{i=1}^n \frac{1}{1-(x_{\pi_1} + \cdots + x_{\pi_{i-1}})}.
\end{equation}
\end{theorem}

Despite the fact that by Proposition~\ref{proposition.length reduced} the right Cayley graph is shallow in the sense
that each vertex is at most $n-1$ steps away from the minimal ideal and the existence of an explicit formula for the stationary
distribution, this is not enough to give a tight bound on the mixing time. The reason is that the expression for
$\Psi_\pi$ does not have the property required in Theorems~\ref{theorem.main} and~\ref{theorem.main1}
that each term of degree $\ell$ in its formal power sum expansion corresponds to a semaphore code word $s$ of 
length $\ell$. Furthermore, the $\mathscr{R}$-classes (or strongly connected components) in the right Cayley graph 
can become very big, especially when $P$ has a maximal element. This makes it hard to analyze the mixing time for the 
promotion Markov chain in general. Here we propose a new Markov chain on linear extensions of a poset
which gives rise to an $\mathscr{R}$-trivial semigroup (where all strongly connected components have size one).

\subsubsection{A variant of the promotion Markov chain}
As before let $P$ be a poset with $n$ elements and $\mathcal{L}(P)$ the set of linear extensions of $P$.
Denote by $\mathcal{W}(P)$ the set of subwords of linear extensions in $\mathcal{L}(P)$ and set $A=[n]$. 
We define a semigroup on $\mathcal{W}(P)$ as follows. Let $w \in \mathcal{W}(P)$ and $a\in A$. Then define
\begin{equation}
\label{equation.product}
	wa = \begin{cases} w & \text{if $a\in w$,}\\
	\mathsf{straight}(wa) & \text{if $a \not \in w$.}
	\end{cases}
\end{equation}
Here $\mathsf{straight}(wa)$ is defined as follows. If $wa$ is a subword of a linear extension of $P$, then
$\mathsf{straight}(wa)=wa$. If not, write $w=w_1\ldots w_k$ and find the largest $1\leqslant j_1\leqslant k$
such that $a \prec w_{j_1}$ in $P$. Interchange $w_{j_1}$ and $a$. Repeat by finding the largest $1\leqslant j_2 <j_1$
such that $a \prec w_{j_2}$. Interchange $w_{j_2}$ and $a$. Repeat until no further element bigger than $a$ exists to the
left. The result is $\mathsf{straight}(wa)$.

\begin{example}
Take the poset $P$ of Example~\ref{example.promotion}, $w=234 \in \mathcal{W}(P)$, and $a=1$. We have
$1 \prec 4$, so $j_1=3$. Both $2$ and $3$ are incomparable to $1$, so we find $\mathsf{straight}(wa) = 2314
\in \mathcal{L}(P)$.
\end{example}

\begin{lemma}
\label{lemma.straighten}
Let $a\in A$ and $w\in \mathcal{W}(P)$ such that $a\not \in w$. Then $\mathsf{straight}(wa) \in \mathcal{W}(P)$.
\end{lemma}

\begin{proof}
Since $j_1$ is largest such that $a \prec w_{j_1}$, either $w_j \prec a$ or $w_j$ and $a$ are incomparable for 
$j_1<j\leqslant k$. If $w_j \prec a$ by transitivity we find that $w_j \prec w_{j_1}$ which contradicts the fact that 
$w \in \mathcal{W}(P)$. Hence $a$ is incomparable with $w_j$ for all $j_1<j\leqslant k$. 
Suppose $w_{j_1} \prec w_j$ for some $j_1<j\leqslant k$. Then again by transitivity, we have $a\prec w_j$.
This contradicts the maximality of $j_1$. Hence $w_{j_1}$ is incomparable to $w_j$ for all $j_1<j\leqslant k$.
Therefore $a w_{j_1+1} \cdots w_k w_{j_1} \in \mathcal{W}(P)$. Repeating similar arguments for the next segments
(interchanging $a$ with $w_{j_2}$ etc), we find $\mathsf{straight}(wa) \in \mathcal{W}(P)$.
\end{proof}

\begin{proposition}
The set $\mathcal{W}(P)$ together with the product defined in~\eqref{equation.product} forms a semigroup.
\end{proposition}

\begin{proof}
Note that by the proof of Lemma~\ref{lemma.straighten}, the letters inbetween any letters that are interchanged
by the product are incomparable to the interchanged letters. By transitivity, if there are three letters that are interchanged,
say $w_i \ldots w_j \ldots w_k$ with $w_k \prec w_j \prec w_i$, it does not matter in which order this is done, the end
results is $w_k \ldots w_j \ldots w_i$. This proves that the product is associative and hence $\mathcal{W}(P)$ is a 
semigroup with the product in~\eqref{equation.product}.
\end{proof}

Let us now define $(\mathcal{W}(P),A)$ to be the semigroup with product~\eqref{equation.product} and generators
$A=[n]$.

\begin{theorem}
The semigroup $(\mathcal{W}(P),A)$ is $\mathcal{R}$-trivial.
\end{theorem}

\begin{proof}
In the product, the length of the word can either stay the same or increase. When the length stays the same, the word does
not change. This proves that $(\mathcal{W}(P),A)$ is $\mathcal{R}$-trivial.
\end{proof}

\begin{example}
The right Cayley graph of $(\mathcal{W}(P),A)$ for the poset of Example~\ref{example.promotion} is given in Figure~\ref{figure.poset}.
\end{example}

\begin{figure}
\begin{center}
\scalebox{0.25}{
\input{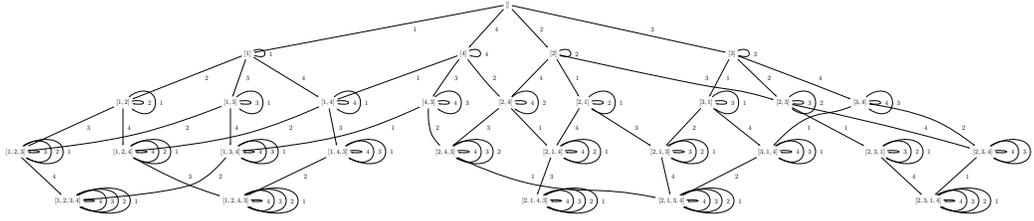}
}
\end{center}
\caption{The right Cayley graph of $(\mathcal{W}(P),A)$ for the poset of Example~\ref{example.promotion}.
\label{figure.poset}}
\end{figure}

Note that the minimal ideal of $(\mathcal{W}(P),A)$ is the set of linear extensions $\mathcal{L}(P)$ of the poset $P$.
Let $\mathcal{M}(\mathcal{W}(P),A)$ be the Markov chain on $\mathcal{L}(P)$ induced by the semigroup
$(\mathcal{W}(P),A)$. More precisely, we transition from $\pi \in \mathcal{L}(P)$ to $a\pi \in \mathcal{L}(P)$ with 
probability $x_a$. 

\begin{proposition}
$\mathcal{M}(\mathcal{W}(P),A)$  is ergodic.
\end{proposition}

\begin{proof}
Note that $\pi \pi'=\pi$ for all $\pi,\pi'\in\mathcal{L}(P)$. Hence the graph of the Markov chain is strongly connected
and hence it is irreducible. Furthermore, if $\pi = \pi_1\ldots \pi_n \in \mathcal{L}(P)$, then $\pi_1 \pi = \pi$, which
means the Markov chain is aperiodic.
\end{proof}

The stationary distribution for $\mathcal{M}(\mathcal{W}(P),A)$ is given by
\[
	\Psi_\pi = \sum_{\substack{\sigma \in S_n\\ [\sigma]_{{\mathcal{W}(P)}} = \pi}} 
	\left( \prod_{i=1}^n \frac{x_{\sigma_i}}{1-\sum_{j=1}^{i-1} x_{\sigma_j}} \right) 
	\qquad \text{for all $\pi \in \mathcal{L}(P)$.}
\]

\begin{theorem}
\label{theorem.expected new}
The expected value $E[\tau]$ for $\mathcal{M}(\mathcal{W}(P),A)$ is bounded above by $n \ln(n) + n \gamma$.
\end{theorem}

\begin{proof}
For a word $w \in \mathcal{W}(P)$, its length $|w|=k$ is bounded by $0\leqslant k \leqslant n$.
For a word of length $|w|=k$, there are $n-k$ transition arrows in $\mathsf{RCay}(\mathcal{W}(P),A)$
originating at $w$, given by all the letters that do not appear in $w$. Hence by the same arguments as for the 
Tsetlin library $E[\tau] \leqslant n \ln(n) + n \gamma$.
\end{proof}

\begin{remark}
Note that the Markov chain $\mathcal{M}(\mathcal{W}(P),A)$ is not identical to the promotion Markov chain.
For example, left multiplication by 4 on $2143$ in $(\mathcal{W}(P),\{1,2,3,4\})$ for the poset in
Example~\ref{example.promotion} yields $2143$, whereas we see from Figure~\ref{figure.promotion}
that in the promotion Markov chain $2143$ goes to $1243$ under $\partial_4$. The full Markov chain transition
diagram is given in Figure~\ref{figure.markov new}.
\end{remark}

Theorem~\ref{theorem.expected new} shows that the mixing time for $\mathcal{M}(\mathcal{W}(P),[n])$ is of order 
$O(n \log n)$. Of course, this does not take the computational complexity of computing the 
product~\eqref{equation.product} into account. For a word of length $k$, this involves up to $k$ swaps.

\begin{figure}
\begin{tikzpicture}[>=latex,line join=bevel,]
\node (node_0) at (179.0bp,309.5bp) [draw,draw=none] {$\mathtt{(1, 2, 4, 3)}$};
  \node (node_1) at (78.0bp,235.5bp) [draw,draw=none] {$\mathtt{(2, 1, 4, 3)}$};
  \node (node_2) at (24.0bp,160.5bp) [draw,draw=none] {$\mathtt{(2, 3, 1, 4)}$};
  \node (node_3) at (133.0bp,84.5bp) [draw,draw=none] {$\mathtt{(2, 1, 3, 4)}$};
  \node (node_4) at (133.0bp,9.5bp) [draw,draw=none] {$\mathtt{(1, 2, 3, 4)}$};
  \draw [yellow,->] (node_0) ..controls (212.86bp,313.04bp) and (221.0bp,311.93bp)  .. (221.0bp,309.5bp) .. controls (221.0bp,308.02bp) and (217.98bp,307.03bp)  .. (node_0);
  \definecolor{strokecol}{rgb}{0.0,0.0,0.0};
  \pgfsetstrokecolor{strokecol}
  \draw (230.0bp,309.5bp) node {$4$};
  \draw [blue,->] (node_0) ..controls (209.46bp,332.54bp) and (239.0bp,329.7bp)  .. (239.0bp,309.5bp) .. controls (239.0bp,292.07bp) and (216.98bp,287.57bp)  .. (node_0);
  \draw (248.0bp,309.5bp) node {$1$};
  \draw [red,->] (node_0) ..controls (142.83bp,297.53bp) and (128.36bp,291.03bp)  .. (117.0bp,282.5bp) .. controls (106.08bp,274.29bp) and (96.274bp,262.44bp)  .. (node_1);
  \draw (126.0bp,272.5bp) node {$2$};
  \draw [green,->] (node_0) ..controls (180.69bp,288.18bp) and (182.38bp,264.62bp)  .. (183.0bp,244.5bp) .. controls (184.86bp,184.06bp) and (174.31bp,167.87bp)  .. (150.0bp,112.5bp) .. controls (148.4bp,108.85bp) and (146.4bp,105.1bp)  .. (node_3);
  \draw (191.0bp,198.5bp) node {$3$};
  \draw [blue,->] (node_1) ..controls (111.86bp,248.66bp) and (124.65bp,254.95bp)  .. (135.0bp,262.5bp) .. controls (146.87bp,271.16bp) and (158.16bp,283.33bp)  .. (node_0);
  \draw (164.0bp,272.5bp) node {$1$};
  \draw [yellow,->] (node_1) ..controls (111.86bp,239.04bp) and (120.0bp,237.93bp)  .. (120.0bp,235.5bp) .. controls (120.0bp,234.02bp) and (116.98bp,233.03bp)  .. (node_1);
  \draw (129.0bp,235.5bp) node {$4$};
  \draw [red,->] (node_1) ..controls (108.46bp,258.54bp) and (138.0bp,255.7bp)  .. (138.0bp,235.5bp) .. controls (138.0bp,218.07bp) and (115.98bp,213.57bp)  .. (node_1);
  \draw (147.0bp,235.5bp) node {$2$};
  \draw [green,->] (node_1) ..controls (63.098bp,214.8bp) and (47.065bp,192.54bp)  .. (node_2);
  \draw (67.0bp,198.5bp) node {$3$};
  \draw [green,->] (node_2) ..controls (57.864bp,164.24bp) and (66.0bp,163.06bp)  .. (66.0bp,160.5bp) .. controls (66.0bp,158.94bp) and (62.979bp,157.89bp)  .. (node_2);
  \draw (75.0bp,160.5bp) node {$3$};
  \draw [red,->] (node_2) ..controls (54.459bp,183.78bp) and (84.0bp,180.91bp)  .. (84.0bp,160.5bp) .. controls (84.0bp,142.8bp) and (61.785bp,138.29bp)  .. (node_2);
  \draw (93.0bp,160.5bp) node {$2$};
  \draw [yellow,->] (node_2) ..controls (39.356bp,141.3bp) and (54.749bp,123.86bp)  .. (71.0bp,112.5bp) .. controls (79.712bp,106.41bp) and (89.962bp,101.14bp)  .. (node_3);
  \draw (80.0bp,122.5bp) node {$4$};
  \draw [blue,->] (node_2) ..controls (39.577bp,130.73bp) and (69.902bp,75.769bp)  .. (104.0bp,36.5bp) .. controls (107.43bp,32.552bp) and (111.38bp,28.569bp)  .. (node_4);
  \draw (83.0bp,84.5bp) node {$1$};
  \draw [yellow,->] (node_3) ..controls (136.5bp,109.12bp) and (139.38bp,143.9bp)  .. (129.0bp,170.5bp) .. controls (121.56bp,189.57bp) and (106.56bp,207.58bp)  .. (node_1);
  \draw (143.0bp,160.5bp) node {$4$};
  \draw [green,->] (node_3) ..controls (118.93bp,103.85bp) and (104.67bp,121.37bp)  .. (89.0bp,132.5bp) .. controls (79.491bp,139.26bp) and (68.095bp,144.85bp)  .. (node_2);
  \draw (119.0bp,122.5bp) node {$3$};
  \draw [red,->] (node_3) ..controls (166.86bp,91.982bp) and (175.0bp,89.621bp)  .. (175.0bp,84.5bp) .. controls (175.0bp,81.379bp) and (171.98bp,79.284bp)  .. (node_3);
  \draw (184.0bp,84.5bp) node {$2$};
  \draw [blue,->] (node_3) ..controls (133.0bp,64.368bp) and (133.0bp,43.594bp)  .. (node_4);
  \draw (142.0bp,46.5bp) node {$1$};
  \draw [yellow,->] (node_4) ..controls (190.83bp,21.837bp) and (248.0bp,41.476bp)  .. (248.0bp,84.5bp) .. controls (248.0bp,235.5bp) and (248.0bp,235.5bp)  .. (248.0bp,235.5bp) .. controls (248.0bp,262.0bp) and (224.03bp,282.97bp)  .. (node_0);
  \draw (257.0bp,160.5bp) node {$4$};
  \draw [green,->] (node_4) ..controls (145.33bp,23.178bp) and (149.85bp,29.724bp)  .. (152.0bp,36.5bp) .. controls (154.68bp,44.974bp) and (154.61bp,48.003bp)  .. (152.0bp,56.5bp) .. controls (150.83bp,60.319bp) and (148.94bp,64.084bp)  .. (node_3);
  \draw (162.0bp,46.5bp) node {$3$};
  \draw [red,->] (node_4) ..controls (118.42bp,23.304bp) and (113.43bp,29.664bp)  .. (111.0bp,36.5bp) .. controls (108.02bp,44.875bp) and (108.09bp,48.099bp)  .. (111.0bp,56.5bp) .. controls (112.37bp,60.462bp) and (114.56bp,64.29bp)  .. (node_3);
  \draw (120.0bp,46.5bp) node {$2$};
  \draw [blue,->] (node_4) ..controls (166.86bp,16.588bp) and (175.0bp,14.352bp)  .. (175.0bp,9.5bp) .. controls (175.0bp,6.5436bp) and (171.98bp,4.5583bp)  .. (node_4);
  \draw (184.0bp,9.5bp) node {$1$};
\end{tikzpicture}
\caption{The Markov chain $\mathcal{M}(\mathcal{W}(P),[4])$ for the poset of Example~\ref{example.promotion}.
\label{figure.markov new}}
\end{figure}
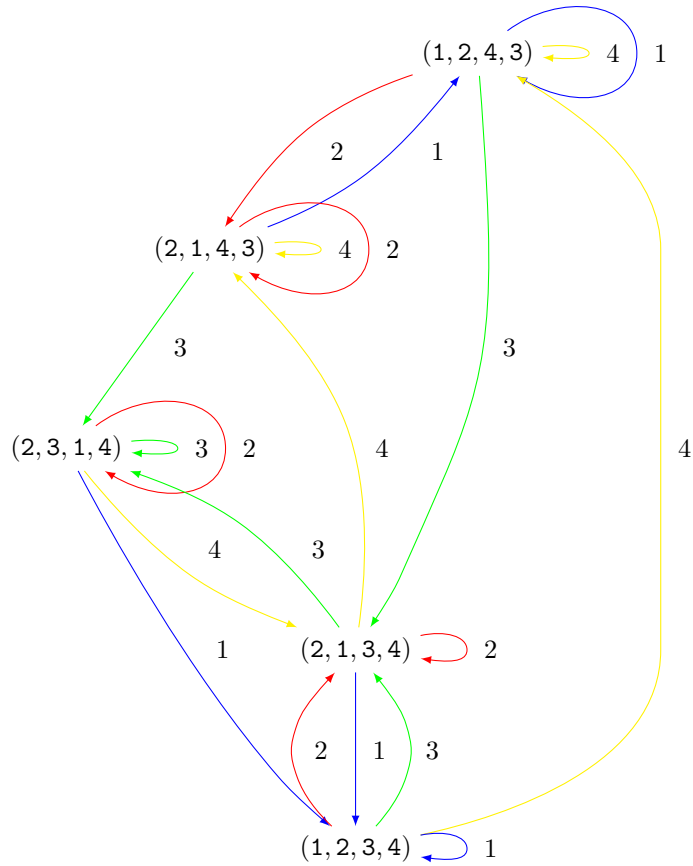

\bibliographystyle{plain}
\bibliography{mixing_time}{}

\end{document}